\documentclass[reqno]{amsart}

\usepackage{amssymb,amsmath,amsthm,xcolor,enumerate,comment,hyperref,cleveref}
\usepackage[utf8]{inputenc}

\allowdisplaybreaks

\newtheorem{thrm}{Theorem}[section]
\newtheorem{cor}[thrm]{Corollary}
\newtheorem{lem}[thrm]{Lemma}
\newtheorem{prop}[thrm]{Proposition}

\theoremstyle{definition}
\newtheorem{defn}[thrm]{Definition}
\newtheorem{exm}[thrm]{Example}

\newtheorem{rem}[thrm]{Remark}

\crefrangeformat{equation}{#3(#1)#4--#5(#2)#6}

\crefname{thrm}{Theorem}{Theorems}
\crefname{lem}{Lemma}{Lemmas}
\crefname{cor}{Corollary}{Corollaries}
\crefname{prop}{Proposition}{Propositions}
\crefname{defn}{Definition}{Definitions}
\crefname{exm}{Example}{Examples}
\crefname{rem}{Remark}{Remarks}
\crefname{section}{Section}{Sections}
\crefname{equation}{\unskip}{\unskip}
\crefname{enumi}{\unskip}{\unskip}

\makeatletter
\newcommand{\mylabel}[2]{#2\def\@currentlabel{#2}\label{#1}}
\makeatother

\renewcommand{\iff}{\Leftrightarrow}

\newcommand{\id}{\mathrm{id}}

\newcommand{\cD}{\mathcal D}
\newcommand{\U}[1]{\mathcal U{(#1)}}
\newcommand{\cF}{\mathcal F}
\newcommand{\cC}{\mathcal C}

\newcommand{\End}[1]{\operatorname{\mathrm{End}}{#1}}

\DeclareMathOperator{\im}{im}

\newcommand{\pr}{\mathrm{pr}}

\newcommand{\m}{{}^{-1}}

\newcommand{\0}{\theta}
\newcommand{\e}{\varepsilon}

\newcommand{\af}{\alpha}
\newcommand{\bt}{\beta}
\newcommand{\lb}{\lambda}
\newcommand{\Lb}{\Lambda}

\newcommand{\D}{\mathcal D}
\newcommand{\f}{\varphi}
\newcommand{\s}{\sigma}
\newcommand{\vt}{\vartheta}
\newcommand{\dl}{\delta}

\newcommand{\A}{\mathcal A}
\newcommand{\R}{{\mathcal R}}
\newcommand{\B}{\mathcal B}
\newcommand{\M}{\mathcal M}

\newcommand{\mt}{\mapsto}

\newcommand{\tl}{\tilde}
\newcommand{\wtl}{\widetilde}

\newcommand{\sst}{\subseteq}
\newcommand{\ol}{\overline}

\begin{document}

\title{Globalization of partial cohomology of groups}

\author{Mikhailo Dokuchaev}
\address{Insituto de Matem\'atica e Estat\'istica, Universidade de S\~ao Paulo,  Rua do Mat\~ao, 1010, S\~ao Paulo, SP,  CEP: 05508--090, Brazil}
\email{dokucha@gmail.com}

\author{Mykola Khrypchenko}
\address{Departamento de Matem\'atica, Universidade Federal de Santa Catarina, Campus Reitor Jo\~ao David Ferreira Lima, Florian\'opolis, SC,  CEP: 88040--900, Brazil}
\email{nskhripchenko@gmail.com}

\author{Juan Jacobo Sim\'on}
\address{Departamento de Matem\'{a}ticas, Universidad de Murcia, 30071 Murcia, Espa\~{n}a}
\email{jsimon@um.es} 


\thanks{This work was partially supported by CNPq of Brazil (Proc. 305975/2013-7),  FAPESP of Brazil (Proc. 2012/01554-7, 2015/09162-9), MINECO (MTM2016-77445-P) and Fundaci\'on S\'eneca of Spain}

\begin{abstract}
 We study the relations between partial and global group cohomology with values in a commutative unital ring $\A$. In particular, for a unital partial action of a group $G$ on $\A$, such that $\A$ is a direct product of commutative indecomposable rings, we show that any partial $n$-cocycle of $G$ with values in $\A$ is globalizable.
\end{abstract}

\maketitle

\section*{Introduction}\label{sec-intro}

Given a partial  action it is natural to ask whether there exists a global action which restricts to the partial one. This question was first considered in the PhD Thesis \cite{Abadie} (see also \cite{AbadieTwo}) and independently in \cite{St2} and \cite{KL} for partial group actions, with subsequent developments in \cite{AbDoExSi,BeF,CF2,CFMarcos,DdRS,DE,DES2,EGG,F,KN16,Pi4}.  More generally the problem was investigated for partial semigroup actions in \cite{GouHol1,Hol1,Kud,Megre2}, for partial groupoid actions in \cite{BP,BPi,Gil} and in the context of partial (weak) Hopf (co)actions in \cite{AAB,AB2,AB3,ABDP2,CasPaqQuaSant,CasPaqQuaSant2,CasQua}.

Globalization results help one to use known facts on global actions in the studies involving partial ones. Thus the first purely ring theoretic  globalization fact \cite[Theorem 4.5]{DE} stimulated intensive algebraic activity, permitting, in particular, to develop a Galois Theory of commutative rings \cite{DFP}. The latter, in its turn, inspired the definition and study  of the concept of a partial action of  a Hopf algebra in \cite{CaenJan}, which is based on globalizable partial group actions, and which became a starting point for  interesting Hopf theoretic developments. Moreover, globalizable partial actions are more ma\-na\-ge\-a\-ble, so that  the great majority of ring theoretic studies on the subject deal with the globalizable case. Among the recent applications of globalization facts we mention their remarkable use  to paradoxical decompositions in \cite{AraE} and to restriction semigroups in \cite{Kud}.   The reader is referred to the surveys \cite{D2,D3,F2} and to the recent book by R.~Exel \cite{E6} for more information about partial actions and their applications.

In \cite{E0} R.~Exel introduced the general concept of a continuous twisted partial action of a locally compact group on a $C^*$-algebra and proved that any second countable $C^*$-algebraic bundle, which is regular in a certain sense, is isomorphic to the $C^*$-algebraic bundle constructed from a twisted partial group action. The purely algebraic version of this result was obtained in \cite{DES1}. The concept involves a twisting which satisfies a kind of $2$-cocycle equality needed for associativity purposes. Thus, it was natural to work out a cohomology theory, encompassing such twistings, and this was done in \cite{DK}. The partial cohomology from
\cite{DK} is strongly related to H.~Lausch's cohomology of inverse semigroups \cite{Lausch} and nicely fits the theory of partial projective group representations developed in \cite{DN}, \cite{DN2} and \cite{DoNoPi}.

The main globalization result from \cite{DES2} says that if $\A$ is a (possibly infinite) product of indecomposable rings (blocks), then any unital twisted partial action $\af$ of a group $G$ on $\A$ possesses an enveloping action, i.e. there exists a twisted global action $\bt$ of $G$ on a ring $\B$ such that $\A$ can be identified with a two-sided ideal in $\B$, $\af$ is the restriction of $\bt$ to $\A$ and $\B=\sum_{g\in G}\bt_g(\A)$. Moreover, if $\B$ has an identity element, then any two globalizations of $\af$ are equivalent in a   natural sense. If $\A$ is commutative, then $\af$ splits into two parts: a unital partial $G$-module structure on $\A$ (i.e. a unital partial action of $G$ on $\A$) and a twisting which is a partial $2$-cocycle $w$ of $G$ with values in the partial module $\A$.  In this case $\B$ is also commutative, and $\bt$ splits into a global action of $G$ on $\B$ (so we have a global $G$-module structure on $\B$) and a usual $2$-cocycle of $G$ with values in the group of units of the multiplier algebra of $\B$. The above mentioned results from \cite{DES2} mean in this context that given a unital $G$-module structure on $\A$, for any $2$-cocycle of $G$ with values in $\A$ there exists a (usual) $2$-cocycle $u$ of $G$ related to the global action on $\B$ such that $w$ is the restriction of $u$. In this case we say that $u$ is a globalization of $w$ (see \cref{def:glob_cocycle}). Moreover, if $\B$ has an identity element, then any two globalizations of $w$ are cohomologous.

The purpose of the present article is to extend the results from \cite{DES2} in the commutative case to arbitrary $n$-cocycles. The technical difficulties coming from \cite{DES2} are being overcome by improvements and notation. In \cref{sec-background}  we recall some notions needed in the sequel. The main result of \cref{notion-glob} is \cref{w-glob-iff-exists-tilde-w}, in which we prove that given a unital partial $G$-module structure on a commutative ring $\A$, a partial $n$-cocycle $w$ with values in $\A$ is globalizable if and only if $w$ can be extended to an $n$-cochain $\tl w$ of $G$ with values in the unit group $\U\A$ which satisfies a ``more global'' $n$-cocycle identity \cref{cocycle-ident-for-tilde-w}. This is the $n$-analogue of \cite[Theorem~4.1]{DES2} in the commutative setting. The technical part of our work is concentrated in \cref{sec-lemmas}, in which we assume that $\A$ is a product of blocks, and this assumption is maintained for the rest of the paper. Our goal is to construct a more 
manageable partial $n$-cocycle $w'$ which is cohomologous to $w$ (see \cref{w'-cohom-w}). In \cref{existence-glob} we prove our main existence result \cref{glob-exists}. The defining formula for $w'$ permits us to extend easily $w'$ to an $n$-cochain $\wtl{w'}: G^n\to\U\A$ which satisfies our  ``more glo\-bal'' $n$-cocycle identity (see \cref{w'-tilde-is-quasi-cocycle}). Modifying $\wtl{w'}$ by a ``co-boundary looking'' function we define in  \cref{tilde-w-in-terms-of-tl-w'} a function $\tl w: G^n\to\U\A$ and show that $\tl w$ is a desired extension of $w$ fitting \cref{w-glob-iff-exists-tilde-w}, and permitting us to conclude that $w$ is glo\-ba\-li\-za\-ble. The uniqueness of a globalization is treated in \cref{sec-uniqueness}. It turns out that it is possible to omit the assumption that the ring $\B$ under the global action has an identity element, imposed in \cite{DES2} (with $n=2$). More precisely,  we prove in \cref{uniqueness-of-glob} that given a  globalizable partial action $\af$ of 
$G$ on a commutative ring $\A$, which is a product of blocks, and a partial $n$-cocycle $w$ related to $\af$, any two globalizations of $w$ are cohomologous. More generally, arbitrary glo\-ba\-li\-za\-tions of cohomologous partial $n$-cocycles are also cohomologous. This results in \cref{H^n(G_A)-cong-H^n(G_B)} which establishes an isomorphism between the partial cohomology group $H^n(G,\A)$ and the global one $H^n(G,\U{\M (\B )}),$ where  $\U{\M (\B })$ stands for the unit group of the multiplier ring $\M(\B)$ of $\B$. \cref{sec-exm} serves as a demonstration of our technique. In \cref{exm-w-of-Z_3} we give an explicit construction of a globalization of an arbitrary partial $2$-cocycle associated with a ``shift'' partial action of a group of order $3$ on the direct product of $2$ copies of a commutative unital ring. In \cref{H^2(Z_3_A)-triv} we also show (independently from the result of \cref{exm-w-of-Z_3}) that the corresponding partial and global $2$-cohomology groups are isomorphic.

\section{Background on globalization and cohomology of partial actions}\label{sec-background}

In all what follows $G$ will stand for an arbitrary group whose identity element will be denoted by $1,$ and by a ring we shall mean an associative ring, which is not unital in general. Nevertheless,  our main attention will be paid to partial actions on  commutative  and unital rings.

In this section we recall a couple of concepts around partial actions.
\begin{defn}[see \cite{DE}]\label{def:ParAc}
Let  $\A$ be a ring.  A {\it partial action} $\af$ of $G$ on $\A$ is a collection of  two-sided ideals $\D_g \sst \A$ $(g \in G)$ and ring isomorphisms $\af_g:\D_{g\m}\to\D_g$ such that 

\begin{enumerate}
 \item $\D_1 = \A$ and $\af_1$ is the identity automorphism of $\A$;\label{D_1-is-A}
 \item for all $g,h\in G$: $\af_g(\D_{g\m}\cap\D_h)=\D_g\cap\D_{gh}$;\label{af(D_g-inv-cap-D_h)}
 \item for all $g,h\in G$ and $a\in\D_{h\m}\cap\D_{h\m g\m}$: $\af_g\circ\af_h(a)=\af_{gh}(a)$.\label{af_g-circ-af_h}
\end{enumerate}

\end{defn} An equivalent form to state \cref{D_1-is-A,af(D_g-inv-cap-D_h),af_g-circ-af_h} is as follows:

\begin{enumerate}
 \item $\af_1={\rm id}_\A$;
 \item[(iv)] for all $g,h\in G$ and $a\in A$: if $\af_h(a)$ and $\af_g\circ\af_h(a)$ are defined, then $\af_{gh}(a)$ is defined and $\af_g\circ\af_h(a)=\af_{gh}(a)$.
\end{enumerate}


Partial actions can be obtained as restrictions of \textit{global} ones, i.e. those satisfying $\D_g=\A$ for all $g\in G$, as follows. Let $\bt$ be a global action of $G$ on a ring $\B$ and $\A$ a two-sided ideal in $\B$. Then setting $\D_g=\A\cap\bt_g(\A)$ and denoting by $\af_g$ the restriction of $\bt_g$ to $\D_{g\m}$ for all $g\in G$, we readily see that $\af=\{\af_g:\D _{g\m}\to\D_g\mid g\in G\}$ is a partial action of $G$ on $\A$, called the {\it restriction} of $\bt$ to $\A$, and $\alpha$ is said to be an {\it admissible restriction} of $\bt$ if $\B=\sum_{g\in G}\bt_g(\A)$. Clearly, if $\B\neq\sum_{g\in G}\bt_g(\A)$, then replacing $\B$ by $\sum_{g\in G}\bt_g(\A)$, the partial action $\af$ can be viewed as an admissible restriction. Partial actions isomorphic to restrictions of global ones are called {\it globalizable}. The notion of an isomorphism of partial action is defined as follows. 

\begin{defn}[see p. 17 from \cite{AbadieTwo} and Definition 4 from \cite{DN2}]\label{def:MorfParAc} 
Let $\A$ and $\A'$ be rings and   $\af=\{\af_g:\D_{g\m}\to\D _g\mid g\in G\}$, $\af'=\{\af'_g:\D'_{g\m}\to\D'_g\mid g\in G\}$ be partial actions of $G$ on $\A$ and $\A'$, respectively. A {\it morphism} $(\A,\af)\to(\A',\af')$ of partial actions is a ring homomorphism $\f:\A\to\A'$ such that for any $g\in G$ and $a\in\D_{g\m}$ the next two conditions are satisfied:
\begin{enumerate}
 \item $\f(\D_g)\sst\D'_g$;\label{f(D_g)-sst-D'g}
 \item $\f(\af_g(a))=\af'_g(\f(a))$.\label{f(af_g(a))=af'_g(f(a))}
\end{enumerate}
We say that a morphism $\f:(\A,\af)\to(\A',\af')$ of partial actions is an {\it isomorphism}\footnote{This was called {\it equivalence} in \cite[Definition 4.1]{DE}.} if $\f:\A\to\A'$ is an isomorphism of rings and $\f(\D_g)=\D'_g$ for each $g\in G$.
\end{defn}

By \cite[Theorem 4.5]{DE} a partial action $\alpha$ on a unital ring $\A$ is globalizable exactly when each ideal $\D_g$ is a unital ring, i.e. $\D_g$ is generated by an idempotent which is central in $\A$, and which will be denoted by $1_g$. In order to guarantee the uniqueness of a globalization one considers the following.

\begin{defn}[Definition 4.2 from \cite{DE}]\label{def:envelopeaction} 
A global action $\bt $ of $G$ on a ring $\B$ is said to be an {\it enveloping} action for the partial action $\af$ of $G$ on a ring $\A$ if $\af $ is isomorphic to an admissible restriction of $\bt$. 
\end{defn}

By the above mentioned \cite[Theorem~4.5]{DE}, an enveloping action $\bt$ for a globalizable partial action of $G$ on a unital ring $\A$ is unique up to an isomorphism. Denote by $\cF=\cF(G,\A)$ the ring of functions from $G$ to $\A$, i.e. $\cF$ is the Cartesian product of copies of $\A$ indexed by the elements of $G$. Note that by the proof of \cite[Theorem~4.5]{DE}, the ring under the global action  is a subring $\B$ of $\cF$, and consequently  $\B$ is commutative if and only if $\A$ is.

Every ring is a semigroup with respect to multiplication, and if in \cref{def:ParAc} we assume that $\A$ is a (multiplicative) semigroup and the maps $\af_g$ are isomorphisms of semigroups satisfying \cref{D_1-is-A,af(D_g-inv-cap-D_h),af_g-circ-af_h}, then we obtain the concept of a partial action of $G$ on a semigroup (see~\cite{DN}). Furthermore, the concept of a morphism of partial actions on semigroups is obtained from \cref{def:MorfParAc} by assuming that $\f:\A\to\A'$ is a homomorphism of semigroups satisfying \cref{f(D_g)-sst-D'g,f(af_g(a))=af'_g(f(a))}. 

Partial cohomology was defined in \cite{DK} as follows. Let $\af=\{\af_g:\D_{g\m}\to\D_g\mid g \in G\}$ be a partial action of $G$ on a commutative monoid $\A$.  Assume  that each ideal $\D _g$ is unital, i.e. $\D_g$ is generated by an  idempotent $1_g =1_g^\A$. In this case we shall say that $\af $ is a {\it unital} partial action. Then $\D_g \cap \D_h=\D_g\D_h$, for all $g,h\in G,$ so the properties \cref{af(D_g-inv-cap-D_h),af_g-circ-af_h} from \cref{def:ParAc} can be replaced by
\begin{enumerate}
 \item[\mylabel{af(D_g-inv-D_h)}{(ii')}] $\af_g(\D_{g\m}\D_h)=\D_g\D_{gh}$;
 \item[\mylabel{af_g-circ-af_h-on-product}{(iii')}] $\af_g\circ\af_h=\af_{gh}$ on $\D_{h\m}\D_{h\m g\m}$.
\end{enumerate}
Note also that \ref{af_g-circ-af_h-on-product} implies a more general equality 
\begin{align}\label{product-ideals}
 \af_x(\D_{x\m}\D_{y_1}\dots\D_{y_n})=\D_x\D_{xy_1}\dots\D_{xy_n},
\end{align}
for any $x,y_1,\dots,y_n\in G$, which easily follows by observing that $\D_{x\m}\D_{y_1}\dots\D_{y_n}=\D_{x\m}\D_{y_1}\dots\D_{x\m}\D_{y_n}$.

\begin{defn}[see \cite{DK}]\label{defn-p-mod}
A commutative monoid $\A$ with a unital partial action $\af$ of $G$ on $\A$ will be called {\it a (unital) partial $G$-module}. A {\it morphism} of (unital) partial $G$-modules $\f:(\A,\af)\to(\A',\af')$ is a  morphism of partial actions such that its restriction to each $\D_g$ is a homomorphism of monoids $\D_g\to\D'_g$, $g\in G$.
\end{defn}

For simplicity, we shall often omit $\af$ from the pair $(\A,\af)$, if no confusion arises.

\begin{defn}[see \cite{DK}]\label{defn-cochain}
 Let $\A$ be a partial $G$-module and $n$ a positive integer. {\it An $n$-cochain} of $G$ with values in $\A$ is a function $f:G^n\to\A$, such that $f(x_1,\dots,x_n)$ is an invertible element of the ideal $\D_{(x_1,\dots,x_n)}=\D_{x_1}\D_{x_1x_2}\dots \D_{x_1\dots x_n}$ for any $(x_1,\dots,x_n)\in G^n$. By {\it a $0$-cochain} we shall mean an invertible element of $\A,$ i.e. $a\in\U\A$, where  $\U\A$ stands for the group of invertible elements of $\A$.
\end{defn}

Denote the set of $n$-cochains by $C^n(G,\A)$. It is an abelian group under the pointwise multiplication. Indeed, its identity is $e_n$ which is the $n$-cochain defined by 
$$
 e_n(x_1,\dots,x_n)=1_{(x_1,\dots,x_n)}:=1_{x_1}1_{x_1x_2}\dots 1_{x_1\dots x_n},
$$
and the inverse of $f\in C^n(G,\A)$ is $f\m(x_1,\dots,x_n)=f(x_1,\dots,x_n)\m$, where $f(x_1,\dots,x_n)\m$ means the inverse of $f(x_1,\dots,x_n)$ in $\D_{(x_1,\dots,x_n)}$.

The multiplicative form of the classical coboundary homomorphism now can be adapted to our context by replacing the global action by a partial one, and taking inverse elements in the corresponding ideals, as follows.

\begin{defn}[see \cite{DK}]\label{defn-cobound-hom}
 Let $(\A,\af)$ be a partial $G$-module and $n$ a positive integer. For any $f\in C^n(G,\A)$ and $x_1,\dots,x_{n+1}\in G$ define
\begin{align}\label{delta^n-for-n>0}
(\dl^nf)(x_1,\dots,x_{n+1})&=\af_{x_1}(1_{x\m_1}f(x_2,\dots,x_{n+1}))\notag\\
&\quad\cdot\prod_{i=1}^nf(x_1,\dots,x_ix_{i+1},\dots,x_{n+1})^{(-1)^i}\notag\\
&\quad\cdot f(x_1,\dots,x_n)^{(-1)^{n+1}}. 
\end{align}
If $n=0$ and $a$ is an invertible element of $\A$, we set 
\begin{align}\label{delta^0-def}
 (\dl^0a)(x)=\af_x(1_{x\m}a)a\m.
\end{align}
\end{defn}

According to \cite[Proposition~1.5]{DK} the coboundary map $\dl^n$ is a homomorphism $C^n(G,\A)\to C^{n+1}(G,\A)$ of abelian groups, such that 
\begin{align}\label{delta^2=0}
 \dl^{n+1}\dl^nf=e_{n+2}
\end{align} 
for any $f\in C^n(G,\A)$. As in the classical case one defines the abelian groups of {\it partial $n$-cocycles, $n$-co\-boun\-da\-ries and $n$-cohomologies} of $G$ with va\-lu\-es in $\A$ by setting $Z^n(G,\A)=\ker{\dl^n}$, $B^n(G,\A)=\im{\dl^{n-1}}$ and $H^n(G,\A)=\ker{\dl^n}/\im{\dl^{n-1}}$, $n\ge 1$ ($H^0(G,\A)=Z^0(G,\A)=\ker{\dl^0}$). Then two partial $n$-cocycles which represent the same element of $H^n(G,\A)$ are called {\it cohomologous}.

Taking $n=0$, we see that
\begin{align*}
 H^0(G,\A)&=Z^0(G,\A)=\{a\in\U\A\mid\forall x\in G:\ \af_x(1_{x\m}a)=1_xa\},\\
 B^1(G,\A)&=\{f\in C^1(G,\A)\mid\exists a\in\U\A\ \forall x\in G:\ f(x)=\af_x(1_{x\m}a)a\m\}. 
\end{align*} 
Notice that $H^0(G,\A)$ is exactly the subgroup of $\alpha$-invariants of $\U\A$, as defined (for the case of rings) in \cite[p. 79]{DFP}. In order to relate partial cohomology to twisted partial actions, consider the cases $n=1$ and $n=2$. In the first case we have
$$
 (\dl^1 f)(x,y)=\af_x(1_{x\m}f(y))f(xy)\m f(x)
$$
with $f\in C^1(G,\A)$, so that 
\begin{align*}
 Z^1(G,\A)&=\{f\in C^1(G,\A)\mid\forall x,y\in G:\ \af_x(1_{x\m}f(y))f(x)=1_xf(xy)\},\\
 B^2(G,\A)&=\{f\in C^2(G,\A)\mid \exists g\in C^1(G,\A)\ \forall x,y\in G:\ f(x,y)=\af_x(1_{x\m}g(y))g(xy)\m g(x)\},
\end{align*}
and for $n=2$ 
$$
 (\dl^2f)(x,y,z)=\af_x(1_{x\m}f(y,z))f(xy,z)\m f(x,yz) f(x,y)\m,
$$
with $f\in C^2(G,\A)$, and 
$$
 Z^2(G,\A)=\{f\in C^2(G,\A)\mid\forall x,y,z \in G:\ \af_x(1_{x\m}f(y,z))f(x,yz)=f(x,y)f(xy,z)\}.
$$
Now, a unital twisted partial action (see~\cite[Def.~2.1]{DES1}) of $G$ on a commutative ring $\A$ splits into two parts: a unital partial action of $G$ on $\A$, and a twisting which, in our terminology, is a $2$-cocycle with values in the partial $G$-module $\A$. Furthermore, the concept of equivalent unital twisted partial actions from~\cite[Def. 6.1]{DES2} is exactly the notion of equivalence of partial $2$-cocycles.  

We shall use multipliers in order to define globalization of partial cocycles, and for this purpose we remind the reader that the {\it multiplier ring} of an associative not necessarily unital ring $\A$ is the set 
$$
\M(\A)=\{(R,L)\in\End( _\A\A)\times\End(\A_\A):(aR)b = a(Lb)\mbox{ for all } a,b\in\A\}
$$ 
with component-wise addition and multiplication (see \cite{AHdR} or \cite{DE} for more details). Here we use the right-hand side notation for homomorphisms of left $\A$-modules, whereas for homomorphisms of right modules the usual notation is used. Thus given $R\in\End( _\A\A)$,  $L\in\End(\A_\A)$ and $a\in\A$, we write $a\mapsto aR$ and $a\mapsto La$. For a multi\-plier $u=(R,L)\in \M(\A)$ and an element $a\in\A$ we set $au=aR$ and $ua=La$, so that the associativity equality $(au)b=a(ub)$ always holds with $a,b\in\A$.

Notice that 
\begin{align}\label{e-comm-with-u}
eu = ue
\end{align} 
for any $u\in\M(\A)$ and any central idempotent $e\in\A$.  For 
$$
eu=(e^2)u=e(eu)=(eu)e=e(ue)=(ue)e=ue.
$$

Any $a\in\A$ determines a multiplier $u_a$ by setting $u_ab = ab$ and $bu_a= ba$, $b\in\A$, so that $a\mapsto u_a$ gives the canonical homomorphism $\A\to\M(\A)$, which is an isomorphism if $\A$ has $1_\A$ (in this case the inverse isomorphism is given by $\M(\A)\ni u\mapsto u1_\A=1_\A u\in\A$). According to \cite{DE} a ring $\A$ is said to be {\it non-degenerate} if the canonical map $\A \to\M(\A)$ is injective. This is guaranteed if $\A$ is left (or right) {\it $s$-unital}, i.e. for any $a\in\A$ one has $a\in\A a$ (respectively, $a\in a\A$).

Furthermore, given a ring isomorphism $\phi:\A\to\A'$, the map $\M(\A)\ni u\mt\phi u\phi\m\in\M(\A')$, where $\phi u\phi\m=(\phi\m R\phi,\phi L \phi\m)$, $u=(R,L)$, is an isomorphism of rings. In particular, an automorphism $\phi$ of $\A$ gives rise to an automorphism 
\begin{align}\label{conjugation-by-phi} 
u\mt\phi u\phi\m 
\end{align} 
of $\M(\A)$.

We shall also use the following.
\begin{rem}[see Remark 5.2 from \cite{DK2} and Lemma 3.1 from \cite{DK3}]\label{M(A)-comm} 
If $\A$ is a commutative idempotent ring, then $\M(\A)$ is also commutative and for each $w\in \M(\A)$ and $a\in\A$ one has $aw=wa$.
\end{rem} 


\section{The notion of a globalization of a partial cocycle and its relation with an extendibility property}\label{notion-glob}

In this section we introduce the concept of a globalization of a partial $n$-cocycle with values in a commutative unital ring $\A$ and show that a partial $n$-cocycle $w$ is globalizable, provided that an extendibility property for $w$ holds. We start with a general auxiliary result which does not involve partial actions.

Let $G$ be a group and $\A$ a commutative unital ring. For $f\in\cF=\cF(G,\A)$ denote by $f|_t$ the value $f(t)$ and define $\beta_x:\cF\to\cF$ by 
\begin{align}\label{beta_x(f)_t}
\beta_x(f)|_t=f(x\m t), 
\end{align}
where $x,t\in G$. Then $\beta$ is a global action of $G$ on $\cF$ which was used in \cite{DE} to deal with the globalization problem for partial actions on unital rings.

Let $\wtl w:G^n\to\U\A$ be a function, i.e. $\wtl w$ is an element of the group $C^n(G,\U\A)$ of global (classical) $n$-cochains of $G$ with values in $\U\A$. Define $u:G^n\to \U\cF$ by
\begin{align}\label{u-def-n>0}
 u(x_1,\dots,x_n)|_t&=\wtl w(t\m, x_1,\dots,x_{n-1})^{(-1)^n}\wtl w(t\m x_1, x_2, \dots,x_n)\notag\\
 &\quad\cdot\prod_{i=1}^{n-1}\wtl w(t\m,x_1,\dots,x_ix_{i+1},\dots,x_n)^{(-1)^i},\ n>0.
\end{align}

We proceed with a technical fact which will be used in the main result of this section.

\begin{lem}\label{u-is-n-cocycle}
 The $n$-cochain $u$ is an $n$-cocycle with respect to the action $\beta$ of $G$ on $\U\cF$, i.e. $u\in Z^n(G,\U\cF)$.  
\end{lem}
\begin{proof}
We need to show that the function
\begin{align}\label{delta^n-u}
 \beta_{x_1}(u(x_2,\dots,x_{n+1}))\prod_{i=1}^n u(x_1,\dots,x_ix_{i+1},\dots,x_{n+1})^{(-1)^i}u(x_1,\dots,x_n)^{(-1)^{n+1}}
\end{align}
is the identity, i.e. it equals $1_{\cF}$ for any $x_1,\dots,x_{n+1}\in G$. Evaluating \cref{delta^n-u} at $t$ and using \cref{beta_x(f)_t}, we get
\begin{align}\label{delta^n-u-without-beta}
 u(x_2,\dots,x_{n+1})|_{x\m_1 t}\prod_{i=1}^n u(x_1,\dots,x_ix_{i+1},\dots,x_{n+1})^{(-1)^i}|_tu(x_1,\dots,x_n)^{(-1)^{n+1}}|_t.
\end{align}
Denote by  $\tl\dl^n:C^n(G,\U\A)\to C^{n+1}(G,\U\A)$ the coboundary operator which corresponds to the trivial $G$-module, i.e.
\begin{align}\label{tilde-delta}
 (\tl\dl^n\wtl w)(x_1,\dots,x_{n+1})&=\wtl w(x_2,\dots,x_{n+1})\notag\\
 &\quad\cdot\prod_{i=1}^n\wtl w(x_1,\dots,x_ix_{i+1},\dots,x_{n+1})^{(-1)^i}\notag\\
 &\quad\cdot\wtl w(x_1,\dots,x_n)^{(-1)^{n+1}}.
\end{align}
We see from \cref{u-def-n>0} that
$$
 u(x_1,\dots,x_n)|_t=\wtl w(x_1,\dots,x_n)(\tl\dl^n\wtl w)(t\m,x_1,\dots,x_n)\m.
$$
Therefore, \cref{delta^n-u-without-beta} becomes
\begin{align*}
&\wtl w(x_2,\dots,x_{n+1})(\tl\dl^n\wtl w)(t\m x_1,x_2,\dots,x_{n+1})\m\\
&\cdot\prod_{i=1}^n \wtl w(x_1,\dots,x_ix_{i+1},\dots,x_{n+1})^{(-1)^i}\prod_{i=1}^n(\tl\dl^n\wtl w)(t\m,x_1,\dots,x_ix_{i+1},\dots,x_{n+1})^{(-1)^{i+1}}\\
&\cdot\wtl w(x_1,\dots,x_n)^{(-1)^{n+1}}(\tl\dl^n\wtl w)(t\m,x_1,\dots,x_n)^{(-1)^n}.
\end{align*}
Regrouping the factors and using \cref{tilde-delta}, we obtain
\begin{align*}
&(\tl\dl^n\wtl w)(x_1,\dots,x_{n+1})(\tl\dl^n\wtl w)(t\m x_1,x_2,\dots,x_{n+1})\m\\
&\cdot\prod_{i=1}^n(\tl\dl^n\wtl w)(t\m,x_1,\dots,x_ix_{i+1},\dots,x_{n+1})^{(-1)^{i+1}}(\tl\dl^n\wtl w)(t\m,x_1,\dots,x_n)^{(-1)^n},
\end{align*}
which is $(\tl\dl^{n+1}\tl\dl^n\wtl w)(t\m,x_1,\dots,x_n)=1_\A$.
\end{proof}

Let now $\af$ be a unital partial action of $G$ on $\A$. Then
\begin{align}\label{f(a)|_t=af_t-inv(1_ta)}
 \f(a)|_t=\af_{t\m}(1_ta),
\end{align}
where $t\in G$ and $a\in \A$, defines an embedding of $\A$ into $\cF$, and $(\bt,\B)$ is an enveloping action for $(\af,\A)$, where $\B=\sum_{g\in G}\bt_g(\f(\A))$ (see the proof of \cite[Theorem 4.5]{DE}). Since $(\bt,\B)$ is unique up to an isomorphism, it follows by \cite[Theorem~3.1]{DdRS} that $\B$ is left $s$-unital. Hence there is a canonical embedding of $\B$ into the multiplier ring $\M(\B)$ and, moreover, $\B$ is commutative because $\A$ is. In addition, $\B$ is idempotent because $\B$ is left $s$-unital, which implies that $\M(\B)$ is commutative thanks to \cref{M(A)-comm}. Observe that the global action $\bt$ of $G$ on $\B$ can be extended by \cref{conjugation-by-phi} to a global action $\bt^*$ of $G$ on $\M(\B)$ by setting  
\begin{align}\label{gamma^*}
 \bt^*_g(u)=\bt_g u\bt\m_g =(\bt\m_g R\bt_g,\bt_g L\bt\m_g), 
\end{align}
where $u = (R,L)\in\M(\B)$ and $g\in G$.  Moreover, the commutativity of $\M(\B)$ permits us to consider the group of units $\U{\M(\B)}$ as a $G$-module via $\bt^*$.

\begin{defn}\label{def:glob_cocycle}  
Let $\af=\{\af_g:\D_{g\m}\to\D_g\mid g\in G\}$ be a unital partial action of $G$ on a commutative ring $\A$ and $w\in Z^n(G,\A)$. Denote by $\bt$ the enveloping action of $G$ on $\B$ and by $\f:\A\to\B$ the embedding which transforms $\af$ into an admissible restriction of $\bt$. A {\it globalization} of $w$ is a (classical) $n$-cocycle $u\in Z^n(G,\U{\M(\B)})$, where $G$ acts on  $\U{\M(\B)}$ via $\bt^*$, such that
\begin{align}\label{w-is-restr-of-u}
\f(w(x_1,\dots,x_n))=\f(1_{(x_1,\dots,x_n)})u(x_1,\dots,x_n),
\end{align}
for any $x_1,\dots,x_n\in G$. If $n=0$, then by $1_{(x_1,\dots,x_n)}$ we mean $1_\A$ in \cref{w-is-restr-of-u}.
\end{defn}
Observe from \cref{e-comm-with-u} that \cref{w-is-restr-of-u} implies
$$
 \f(w(x_1,\dots,x_n))=u(x_1,\dots,x_n)\f(1_{(x_1,\dots,x_n)}).
$$ 

 
It is readily seen that if $\B$ contains $1_\B$, then the isomorphism $\M(\B)\cong\B$ transforms $\bt^*$ into $\bt$, and the globalization $u$ is an $n$-cocycle with values in $\U\B$. 

\begin{prop}\label{0-cocycle-is-globalizable}
 Observe that any partial $0$-cocycle $w$ is globalizable, and its glo\-ba\-li\-za\-tion is the constant function $u\in\U\cF$ with $u|_t=w$ for all $t\in G$. Moreover, such $u$ is unique.
\end{prop}
\begin{proof}
	Indeed, \cref{w-is-restr-of-u} reduces to $\f(w)=\f(1_\A)u$, which is the partial $0$-cocycle identity for $w$ by \cref{f(a)|_t=af_t-inv(1_ta)}. Moreover, $u$ is an (invertible) multiplier of $\B$, as
	\begin{align*}
		\bt_g(\f(a))|_tu|_t&=\f(a)|_{g\m t}w=\af_{t\m g}(1_{g\m t}a)w=\af_{t\m g}(1_{g\m t}a)\cdot 1_{t\m g}w\\
		&=\af_{t\m g}(1_{g\m t}aw)=\bt_g(\f(aw))|_t
	\end{align*}
	thanks to \cref{f(a)|_t=af_t-inv(1_ta),beta_x(f)_t} and the $0$-cocycle identity for $w$. Applying $\bt_xu\bt\m_x$ to an arbitrary $\bt_y(\f(a))\in\B$ and evaluating the result at any $t\in G$, we obtain by \cref{beta_x(f)_t}
	\begin{align*}
		(\bt_xu\bt\m_x)(\bt_y(\f(a)))|_t&=\bt_xu\bt_{x\m y}(\f(a))|_t=(u\bt_{x\m y}(\f(a)))|_{x\m t}\\
		&=u|_{x\m t}\bt_{x\m y}(\f(a))|_{x\m t}=w\f(a)|_{y\m t}\\
		&=w\bt_y(\f(a))|_t=u|_t\bt_y(\f(a))|_t,
	\end{align*} 
	so that $\bt_xu\bt\m_x$ coincides with $u$ as a multiplier on $\B$, i.e. $u$ is a (global) $0$-cocycle with respect to the action $\bt^*$ of $G$ on $\M(\B)$.
	
	Now if the restrictions of $u_1,u_2\in H^0(G,\U{\M(\B)})$ to the ideal $\f(\A)$ coincide, then $\f(1_\A)u_1=\f(1_\A)u_2$. Applying $\bt_x$ to this equality and using the $0$-cocycle identity for $u_i$ which means that $\bt_xu_i\bt\m_x=u_i$, $i=1,2$, one gets $\bt_x(\f(1_\A))u_1=\bt_x(\f(1_\A))u_2$ for all $x\in G$. Consequently, $\bt_x(\f(a))u_1=\bt_x(\f(a))u_2$ for all $x\in G$ and $a\in\A$. It follows that $u_1=u_2$, as $\B=\sum_{g\in G}\bt_g(\A)$. In particular, this holds for any two globalizations of the same $w\in H^0(G,\A)$.
\end{proof}

\begin{cor}\label{H^0(G_A)-cong-H^0(G_M(B))}
	We have $H^0(G,\A)\cong H^0(G,\U{\M(\B)})$.
\end{cor}
\begin{proof}
	By \cref{0-cocycle-is-globalizable} there is an injective map from $H^0(G,\A)$ to $H^0(G,\U{\M(\B)})$ sending $w\in H^0(G,\A)$ to its globalization $u\in H^0(G,\U{\M(\B)})$, which is readily seen to be a group homomorphism. It follows from the uniqueness of the globalization that this map is also surjective, since any $u\in H^0(G,\U{\M(\B)})$ is the globalization of its restriction to $\f(\A)$. 
\end{proof}

Given an arbitrary $n>0$, as in the case $n=2$ (see \cite[Theorem 4.1]{DES2}), we are able to reduce the globalization problem for partial $n$-cocycles to an extendibility property.

\begin{thrm}\label{w-glob-iff-exists-tilde-w} 
Let $\af=\{\af_g:\D_{g\m}\to\D_g\mid g\in G\}$ be a unital partial action of $G$ on a commutative ring $\A$ and $w\in Z^n(G,\A)$. Then $w$ is glo\-ba\-li\-za\-ble if and only if there exists a function $\wtl w:G^n\to\U\A$ which satisfies the equalities
\begin{align}\label{cocycle-ident-for-tilde-w}
 \af_{x_1}\left(1_{x\m_1}\wtl w(x_2,\dots,x_{n+1})\right)&\prod_{i=1}^n\wtl w(x_1,\dots,x_ix_{i+1},\dots,x_{n+1})^{(-1)^i}\notag\\
 & \cdot\wtl w(x_1,\dots,x_n)^{(-1)^{n+1}} = 1_{x_1},
\end{align} and 
\begin{align}\label{w-is-restr-of-tilde-w}
 w(x_1,\dots,x_n)=1_{(x_1,\dots,x_n)}\wtl w(x_1,\dots,x_n), 
\end{align} 
for all $x_1,\dots,x_{n+1}\in G$.
\end{thrm}
\begin{proof}
We shall assume that $n>0$, as $n=0$ was considered in \cref{0-cocycle-is-globalizable}. 

Suppose that $w\in Z^n(G,\A)$ is globalizable. Denote by $(\bt,\B)$ an enveloping action of $(\af,\A)$ and let $\bt^*$ be the corresponding  action of $G$ on $\M(\B)$ (see \cref{gamma^*}). Let $u\in Z^n(G,\U{\M(\B)})$ be a globalization of $w$ and  define $\wtl w(x_1,\dots,x_n)\in\U\A$ by 
\begin{align}\label{tilde-w-defn}
 \f(\wtl w(x_1,\dots,x_n))=\f(1_\A)u(x_1,\dots,x_n)=u(x_1,\dots,x_n)\f(1_\A).
\end{align}
Evidently, $\wtl w(x_1,\dots,x_n)\in\U\A$, as $u(x_1,\dots,x_n)$ is an invertible multiplier, and $\f(\wtl w(x_1,\dots,x_n)\m)=\f(1_\A)u\m(x_1,\dots,x_n)=u\m(x_1,\dots,x_n)\f(1_\A)$. Then \cref{w-is-restr-of-tilde-w} clearly holds by \cref{w-is-restr-of-u}, and for \cref{cocycle-ident-for-tilde-w} notice first that 
\begin{align}\label{phi(1_g)-in-terms-of-1_A}  
\f (1_g) = \bt _g (\f(1_{\A})) \f (1_{\A} ),
\end{align}
and consequently (and in fact more generally),
\begin{align}\label{phi(af_g(a1_g-inv))}
 \f(\af_g(1_{g\m}a))=\bt_g(\f(a))\f(1_\A),
\end{align}
for all $g\in G$ and $a\in\A$ (see \cite[p. 79]{DFP}). The (global) $n$-cocycle identity for $u$ is of the form
\begin{align}\label{n-cocycle-ident-for-u}
&\bt^*_{x_1}(u(x_2,\dots,x_{n+1}))
\prod_{i=1}^n u(x_1,\dots,x_ix_{i+1},\dots,x_{n+1})^{(-1)^i}\notag\\
&\cdot u(x_1,\dots,x_n)^{(-1)^{n+1}}=1_{\M(\B)}. 
\end{align}
Applying the first multiplier in \cref{n-cocycle-ident-for-u} to $\f(1_{x_1})$ and using \cref{phi(af_g(a1_g-inv)),gamma^*,tilde-w-defn}, we obtain 
\begin{align*}
\bt^*_{x_1}(u(x_2,\dots,x_{n+1}))\f(1_{x_1})&=(\bt_{x_1}u(x_2,\dots,x_{n+1})\bt\m _{x_1})(\bt_{x_1}(\f(1_\A))\f(1_\A))\\
&=(\bt_{x_1}(u(x_2,\dots,x_{n+1})\f(1_\A)))\f(1_\A)\\
&=(\bt_{x_1}[\f(\wtl w(x_2,\dots,x_{n+1}))])\f(1_\A)\\ 
&=\f(\af_{x_1}( 1_{x\m_1}\wtl w(x_2,\dots,x_{n+1}))).
\end{align*} 
Then applying both sides of \cref{n-cocycle-ident-for-u} to $\f(1_{x_1})$ and using axioms of a multiplier, we readily see that  \cref{cocycle-ident-for-tilde-w} is a consequence of \cref{n-cocycle-ident-for-u}.

Suppose now that there exists $\wtl w:G^n\to\U\A$ such that \cref{cocycle-ident-for-tilde-w,w-is-restr-of-tilde-w} hold. Let $(\beta,\B)$ be the globalization of $(\af,\A)$, with $\beta$, $\B\sst\cF=\cF(G,\A)$ and $\f:\A\to\B$ as described above. In particular, it follows from \cref{product-ideals} that for arbitrary $t,x_1,\dots,x_n$:
\begin{align}\label{phi(e_n)|_t=e_(n+1)}
 \f(1_{(x_1,\dots,x_n)})|_t=1_{(t\m,x_1,\dots,x_n)}.
\end{align}

Taking our $\wtl w$, define $u:G^n\to\U{\cF}$ by formula \cref{u-def-n>0}. We are going to show that $u$ is a globalization of $w$. By \cref{u-is-n-cocycle} one has $u\in Z^n(G,\U{\cF})$. We now check \cref{w-is-restr-of-u}. By \cref{f(a)|_t=af_t-inv(1_ta)}
$$
 \f(w(x_1,\dots,x_n))|_t=\af_{t\m}(1_tw(x_1,\dots,x_n)),
$$
 which by the partial $n$-cocycle identity for $w$ equals
\begin{align*}
 w(t\m x_1, x_2,\dots,x_n)&\prod_{i=1}^{n-1}w(t\m, x_1,\dots,x_ix_{i+1},\dots,x_n)^{(-1)^i}\\
 &\cdot w(t\m,x_1,\dots,x_{n-1})^{(-1)^n}.
\end{align*}
In view of \cref{w-is-restr-of-tilde-w,u-def-n>0,phi(e_n)|_t=e_(n+1)} the latter is
\begin{align*}
1_{(t\m,x_1,\dots,x_n)}u(x_1,\dots,x_n)|_t=\f(1_{(x_1,\dots,x_n)})|_t u(x_1,\dots,x_n)|_t,
\end{align*} 
for arbitrary $t,x_1,\dots,x_n\in G$, proving \cref{w-is-restr-of-u}. 

We proceed with a proof that $u(x_1,\dots,x_n)$ and $u(x_1,\dots,x_n)\m$ are multipliers of $\B$. Notice first that using \cref{cocycle-ident-for-tilde-w} for $(t\m, x_1,\dots,x_n)$  we obtain from \cref{u-def-n>0} that
\begin{align}  
\af_{t\m}(1_t\wtl w(x_1,\dots,x_n))&=1_{t\m}\wtl w(t\m x_1, x_2,\dots,x_n)\notag\\ 
&\quad\cdot\prod_{i=1}^{n-1}\wtl w(t\m,x_1,\dots,x_i x_{i+1},\dots,x_n)^{(-1)^i}\notag\\
&\quad\cdot\wtl w(t\m,x_1,\dots,x_{n-1})^{(-1)^n}\notag\\
&=1_{t\m}u(x_1,\dots,x_n)|_t.\label{af(1_t-inv-tilde-w)} 
\end{align} 
Then by \cref{f(a)|_t=af_t-inv(1_ta)}
\begin{align*}
 u(x_1,\dots,x_n)|_t\f(a)|_t=\af_{t\m}(1_t\wtl w(x_1,\dots,x_n))\af_{t\m}(1_ta),
\end{align*} 
so that
\begin{align}\label{u-phi(a)=phi(a-tilde-w)}
 u(x_1,\dots,x_n)\f(a)=\f(a\wtl w(x_1,\dots,x_n)),
\end{align} 
for all $x_1,\dots,x_n\in G$ and $a\in \A$. Equalities \cref{u-phi(a)=phi(a-tilde-w),u-def-n>0} readily imply 
\begin{align}\label{u-inv-phi(a)=f(a-tilde-w-inv)}
 u(x_1,\dots,x_n)\m\f(a)=\f(a\wtl w(x_1,\dots,x_n)\m).
\end{align} 
Furthermore, applying the $n$-cocycle identity for $u$ to  $(t\m, x_1,\dots x_n)$ we see that 
\begin{align*} 
\bt_{t\m}(u(x_1,\dots,x_n))\f(a)&=u(t\m x_1, x_2,\dots,x_n)\\ 
&\quad\cdot\prod_{i=1}^{n-1}u(t\m,x_1,\dots,x_ix_{i+1},\dots,x_{n+1})^{(-1)^i}\\
&\quad \cdot u(t\m, x_1,\dots,x_{n-1})^{(-1)^n}\f(a),
\end{align*} 
 which belongs to $\f(\A)$  thanks to \cref{u-phi(a)=phi(a-tilde-w),u-inv-phi(a)=f(a-tilde-w-inv)}. Thus $\bt_{t\m}(u(x_1,\dots, x_n))\f(\A)\sst\f(\A)$, which yields $u(x_1,\dots,x_n)\bt_t(\f(\A))\sst\bt_t(\f(\A))$. Since $\B=\sum_{t\in G}\bt_t(\f(\A ))$, it follows that $u(x_1,\dots,x_n)\B\sst\B$, and hence $u(x_1,\dots,x_n)\in\M(\B)$. Similarly, $u(x_1,\dots,x_n)\m\in\M(\B)$, as desired.
 
 	It remains to see that $u\in Z^n(G,\U{\M(\B)})$. Observe that $\M(\cF)\cong\cF$ and $\bt^*$ coincides with $\bt$ up to this isomorphism, so the $n$-cocycle identity for $u$ as an element of $C^n(G,\U{\M(\B)})$ reduces to the $n$-cocycle identity for $u$ as an element of $C^n(G,\U\cF)$. 
\end{proof}

Note that taking $t=1$ in \cref{af(1_t-inv-tilde-w)} we obtain $u(x_1,\dots,x_n)|_1=\wtl w(x_1,\dots,x_n)$.

\section{From \texorpdfstring{$w$}{w} to \texorpdfstring{$w'$}{w'}}\label{sec-lemmas}

Our next purpose is to show that $\wtl w$ in \cref{w-glob-iff-exists-tilde-w} exists, provided that $\A$ is a product of blocks, and we need first some technical preparation for this fact, which we do in the present section. 

Suppose that $\A=\prod_{\lb\in\Lb}\A_\lb$,  where $\A_\lb$ is an indecomposable unital ring, called a {\it block}. So far, we do not assume that $\A$ is commutative. We identify the unity element of $\A_{\mu}$, $\mu\in\Lb$, with the centrally primitive idempotent $1_{\A_\mu}$ of $\A$ which is the function $\Lb\to\bigcup_{\lb \in\Lb}\A_\lb$ whose value at $\mu$ is the identity of $\A_\mu$ and the value at any $\lb\neq\mu$ is the zero of $\A_{\lb}$. Then $\A_\mu$ is identified with the ideal of $\A$ generated by the  idempotent $1_{\A_\mu}$. Denote by $\pr_\mu$ the projection of $\A$ onto $\A_\mu$, namely, $\pr_\mu(a)=1_{\A_\mu}a$. Thus, any $a\in \A$ is identified with the set of its projections $\{\pr_\lb(a)\}_{\lb\in\Lb}$, and we write $a=\prod_{\lb\in\Lb}\pr_\lb(a)$ in this situation. If there exists $\Lb_1\sst\Lb$, such that $\pr_\lb(a)=0 _{\A}$ for all $\lb\in\Lb\setminus\Lb_1$, then we shall also write $a=\prod_{\lb\in\Lb_1}\pr_\lb(a)$, and such elements $a$ form an ideal in $\A$ which we denote by $\prod_{\lb\in\Lb_1}\A_\lb$.

Since $\A_\lb$ is indecomposable, the only central idempotents of $\A_\lb$ are $0_\A$ and $1_{\A_\lb}$. Hence, for any central idempotent $e$ of $\A$ the projection $\pr_\lb(e)$ is either $0_\A$, or $1_{\A_\lb}$. In particular, 
\begin{align}\label{eA=prod_(lb-in-Lb')A_lb}
 e\A=\prod_{\lb\in\Lb_1}\A_\lb,
\end{align}
where $\Lb_1=\{\lb\in\Lb\mid \pr_\lb(e)=1_{\A_\lb}\}$. Thus, the unital ideals of $\A$ are exactly the products of blocks $\A_\lb$ over all $\Lb_1\sst\Lb$.

\begin{lem}\label{isom-principal-ideals}
 Let $I=\prod_{\lb\in\Lb_1}\A_\lb$ and $J=\prod_{\lb\in\Lb_2}\A_\lb$ be unital ideals of $\A$ and $\varphi:I\to J$ an isomorphism. Then there exists a bijection $\s:\Lb_1\to\Lb_2$, such that $\f(\pr_\lb(a))=\pr_{\s(\lb)}(\f(a))$ for all $a\in I$ and $\lb\in\Lb_1$.
\end{lem}
\begin{proof}
 Note that $\{1_{\A_\lb}\}_{\lb\in\Lb_1}$ and $\{1_{\A_\lb}\}_{\lb\in\Lb_2}$ are the sets of centrally primitive idempotents of $I$ and $J$, respectively. Since $\f$ is an isomorphism, $\f(1_{\A_\lb})=1_{\A_{\s(\lb)}}$ for some bijection $\s:\Lb_1\to\Lb_2$. Then 
 $$
  \f(\pr_\lb(a))=\f(1_{\A_\lb}a)=1_{\A_{\s(\lb)}}\f(a)=\pr_{\s(\lb)}(\f(a)).
 $$
\end{proof}

Let $\alpha=\{\alpha_x:\cD_{x\m}\to\cD_x\mid x\in G\}$ be a unital partial action of $G$ on $\A$. By the observation above each ideal $\cD_x$ is a product of  blocks, and  $\alpha_x$ maps a block of $\cD_{x\m}$ onto some block of $\cD_x.$   As in~\cite{DES2} we call $\alpha$ {\it transitive}, when for any pair $\lb',\lb''\in\Lb$ there exists $x\in G$, such that $\A_{\lb'}\sst\cD_{x\m}$ and $\alpha_x(\A_{\lb'})=\A_{\lb''}\sst\cD_x$. 

In all what follows, if otherwise is not stated, we assume that $\alpha$ is transitive. Then we may fix $\lb_0\in\Lb$, so that each $\A_\lb$ is $\alpha_x(\A_{\lb_0})$ for some $x\in G$ with $\A_{\lb_0}\sst\cD_{x\m}$. Observe that, whenever $\A_{\lb_0}\sst\cD_{x\m}\cD_{x'\m}$ and $\alpha_x(\A_{\lb_0})=\alpha_{x'}(\A_{\lb_0})$, it follows that 
$\A_{\lb_0}\sst\cD_{ (x')\m x}$ and $\alpha_{x\m x'}(\A_{\lb_0})=\A_{\lb_0}$. Hence, introducing as in~\cite{DES2} the subgroup
$$
 H=\{x\in G\mid \A_{\lb_0}\sst\cD_{x\m}\mbox{ and }\alpha_x(\A_{\lb_0})=\A_{\lb_0}\}
$$
and choosing a left transversal $\Lb'$ of $H$ in $G$, one may identify $\Lb$ with a subset of $\Lb'$, namely, $\lb\in\Lb$ corresponds to (a unique) $g\in\Lb'$, such that $\A_{\lb_0}\sst\cD_{g\m}$ and $\alpha_g(\A_{\lb_0})=\A_{\lb}$. Assume, moreover, that $\Lb'$ contains the identity element $1$ of $G$. Then $\lb_0$ is identified with $1$ and thus
\begin{align}\label{A_g=af(A_1)}
\A_g=\alpha_g(\A_1)\mbox{ for }g\in\Lb\sst\Lb'.
\end{align}

Given $x\in G$, we use the notation $\bar x$ from~\cite{DES2} for the element of $\Lb'$ such that $x\in\bar xH$. We recall the following useful fact.
\begin{lem}[Lemma 5.1 from~\cite{DES2}]\label{Lb-and-Lb'}
 Given $x\in G$ and $g\in\Lb'$, one has
 \begin{enumerate}
  \item $g\in\Lb\iff \A_1\sst\cD_{g\m}$;\label{g-in-Lb-iff-A_1-in-D_g-inv}
  \item if $g\in\Lb$, then $\ol{xg}\in\Lb\iff \A_g\sst\cD_{x\m}$, and in this situation $\af_x(\A_g)=\A_{\ol{xg}}$.\label{ol-xg-in-Lb-iff-A_g-in-D_x-inv}
 \end{enumerate}
\end{lem}
 Notice that taking $g=1$ in \cref{ol-xg-in-Lb-iff-A_g-in-D_x-inv}, one gets $\ol x\in\Lb\iff\A_1\sst\D_{x\m}$. Then using \cref{ol-xg-in-Lb-iff-A_g-in-D_x-inv} once again, we see that for any $g\in\Lb$
\begin{align}\label{A_g-sst-D_x-iff-A_1-sst-D_g-inv-x}
 \A_g\sst\cD_x\iff\ol{x\m g}\in\Lb\iff\A_1\sst\D_{g\m x}.
\end{align}
In particular, $\A_{\ol{x}}\sst\D_x$ for all $x\in G$, such that $\ol{x}\in\Lb$.

For any $g\in\Lb$ and $a\in \A$ define
\begin{align}\label{0_g-def}
 \0_g(a)=\alpha_g(\pr_1(a)).
\end{align}
Note that by \cref{g-in-Lb-iff-A_1-in-D_g-inv} of \cref{Lb-and-Lb'} the block $\A_1$ is a subset of $\cD_{g\m}$, so $\alpha_g(\pr_1(a))$ makes sense and belongs to $\A_g$. Thus, $\0_g$ is a correctly defined homomorphism\footnote{Observe that this $\0$ differs from the one introduced in~\cite{DES2}. More precisely, denoting $\0$ from~\cite{DES2} by $\0'$, we may write $\0'_g(a)=\0_{g\m}(a)+1_\A-1_{\A_{g\m}}$ for $g\m\in\Lb$.} $\A\to \A_g$. Clearly, 
\begin{align}\label{0_g(a)=0_g(1_xa)}
 \0_g(a)=\0_g(1_xa) 
\end{align}
for any $x\in G$, such that $\A_1\sst\cD_x$. In particular, this holds for $x\in H$ and for $x=g\m$.

Observe also that 
\begin{align}\label{0_g=pr_g-circ-alpha_g}
 \0_g(a)=\0_g(1_{g\m}a)=\pr_g(\alpha_g(1_{g\m}a)) 
\end{align}
in view of \cref{isom-principal-ideals}. It follows that $\0_g(\alpha_{g\m}(1_ga))=\pr_g(1_ga)=\pr_g(a)$, as $\A_g\sst\cD_g$. Therefore,
\begin{align}\label{a=prod-0_g-circ-af_g-inv}
 a=\prod_{g\in\Lb}\0_g(\alpha_{g\m}(1_ga)).
\end{align}

In what follows in this section, we assume $\A$ to be commutative, so that $(\A,\af)$ is a partial $G$-module.
\begin{lem}\label{w=prod-0_g}
 Let $n>0$ and $w\in Z^n(G,\A)$. Then
 \begin{align}
  w(x_1,\dots,x_n)&=\prod_{g\in\Lb}\0_g[w(g\m x_1,x_2,\dots,x_n)\notag\\
 &\quad \quad \cdot\prod_{k=1}^{n-1}w(g\m,x_1,\dots,x_kx_{k+1},\dots,x_n)^{(-1)^k}\notag\\
 &\quad \quad\cdot w(g\m,x_1,\dots,x_{n-1})^{(-1)^n}].\label{w=prod_g-0_g-of-w}
 \end{align}
\end{lem}
\begin{proof}
 By \cref{a=prod-0_g-circ-af_g-inv}
$$
 w(x_1,\dots,x_n)=\prod_{g\in\Lb}\0_g(\alpha_{g\m}(1_gw(x_1,\dots,x_n))).
$$
As $w\in Z^n(G,\A)$, one has
\begin{align*}
  1_{(g\m,x_1,\dots,x_n)}&=(\dl^nw)(g\m,x_1\dots,x_n)\\
 &=\alpha_{g\m}(1_gw(x_1,\dots,x_n))w(g\m x_1,x_2,\dots,x_n)^{-1}\\
 &\quad\cdot\prod_{k=1}^{n-1}w(g\m,x_1,\dots,x_kx_{k+1},\dots,x_n)^{(-1)^{k-1}}\\
 &\quad\cdot w(g\m,x_1,\dots,x_{n-1})^{(-1)^{n-1}}.
\end{align*}
Hence,
\begin{align*}
 \alpha_{g\m}(1_gw(x_1,\dots,x_n))&=w(g\m x_1,x_2,\dots,x_n)\\
 &\quad \cdot\prod_{k=1}^{n-1}w(g\m,x_1,\dots,x_kx_{k+1},\dots,x_n)^{(-1)^k}\\
 &\quad\cdot w(g\m,x_1,\dots,x_{n-1})^{(-1)^n}.
\end{align*}
\end{proof}

Given $x\in G$, denote by $\eta(x)$ the element $x\m\bar x\in H$. Let $n>0$ and $g\in\Lb'$. Define $\eta_n^g:G^n\to H$ by
\begin{align}\label{eta_n^g-def}
 \eta_n^g(x_1,\dots,x_n)=\eta(x\m_n \ol{x\m_{n-1}\dots x\m_1 g})
\end{align}
and $\tau_n^g:G^n\to H^n$ by
\begin{align}\label{tau_n^g-def}
\tau_n^g(x_1,\dots,x_n)=(\eta_1^g(x_1),\eta_2^g(x_1,x_2),\dots,\eta_n^g(x_1,\dots,x_n)). 
\end{align}
 Observe that 
\begin{align}\label{prod-of-etas}
 \eta_1^g(x_1)\eta_2^g(x_1,x_2)\dots\eta_n^g(x_1,\dots,x_n)=\eta(x\m_n\dots x\m_1g)=\eta_1^g(x_1\dots x_n).
\end{align}
 
We shall also need the functions $\s_{n,i}^g:G^n\to G^{n+1}$, $0\le i\le n$, defined by
\begin{align}
 \s_{n,0}^g(x_1,\dots,x_n)&=(g\m,x_1,\dots,x_n),\label{sigma_n0-def}\\
 \s_{n,i}^g(x_1,\dots,x_n)&=(\tau_i^g(x_1,\dots,x_i),(\ol{x\m_i\dots x\m_1 g})\m,x_{i+1},\dots,x_n),\ \ 0<i<n,\label{sigma_ni-def}\\
 \s_{n,n}^g(x_1,\dots,x_n)&=(\tau_{n}^g(x_1,\dots,x_{n}),(\ol{x\m_n\dots x\m_1 g})\m).\label{sigma_nn-def}
\end{align}
In the formulas above we may allow $n$ to be equal to zero, meaning that $\s_{0,0}^g=g\m\in G$.

\begin{defn}\label{defn-of-w'-and-eps}
	 With any $n>0$ and $w\in C^n(G,\A)$ we shall associate 
	\begin{align}
	w'(x_1,\dots,x_n)&=1_{(x_1,\dots,x_n)}\prod_{g\in\Lb}\0_g\circ w\circ\tau_n^g(x_1,\dots,x_n),\label{w'-def}\\
	\e(x_1,\dots,x_{n-1})&=1_{(x_1,\dots,x_{n-1})}\prod_{g\in\Lb}\0_g\left(\prod_{i=0}^{n-1}w\circ\s_{n-1,i}^g(x_1,\dots,x_{n-1})^{(-1)^i}\right).\label{eps-def}
	\end{align}
\end{defn}
 
\begin{lem}\label{w'-and-e-are-cochains}
  Let $n>0$, $w\in C^n(G,\A)$ and $w',\e$ be as in \cref{defn-of-w'-and-eps}. Then $w'\in C^n(G,\A)$ and $\e\in C^{n-1}(G,\A)$.
\end{lem}
\begin{proof}
  Notice by \cref{tau_n^g-def,prod-of-etas}  that
$$
 w\circ\tau_n^g(x_1,\dots,x_n)\in\U{\cD_{\eta(x\m_1 g)}\cD_{\eta(x\m_2 x\m_1 g)}\dots\cD_{\eta(x\m_n\dots x\m_1 g)}}.
$$
Since $\eta(x\m_k\dots x\m_1 g)\in H$, then $\A_1\sst\cD_{\eta(x\m_k\dots x\m_1 g)}$, $1\le k\le n$, so 
$$
 \pr_1\circ w\circ\tau_n^g(x_1,\dots,x_n)\in\U{\A_1}
$$
and hence by \cref{0_g-def}
$$
 \0_g\circ w\circ\tau_n^g(x_1,\dots,x_n)=\alpha_g\circ\pr_1\circ w\circ\tau_n^g(x_1,\dots,x_n)\in\U{\A_g}.
$$
Therefore, the product of the values of $\0_g$ on the right-hand side of \cref{w'-def} belongs to $\U\A$ and thus $w'\in C^n(G,\A)$.

To prove that $\e\in C^{n-1}(G,\A)$ for $n>1,$ observe first that the right-hand side of \cref{eps-def} depends only on $\0_g$ with $g\in\Lb$ satisfying 
\begin{align}\label{A_g-in-D_(x_1_dots_x_n)}
 \A_g\sst\cD_{(x_1,\dots,x_{n-1})}
\end{align}
(if there is no such $g$, then $\cD_{(x_1,\dots,x_{n-1})}$ is zero and thus $\e(x_1,\dots,x_{n-1})$ is automatically invertible in this ideal). Now 
$$
 \prod_{i=0}^{n-1}w\circ\s_{n-1,i}^g(x_1,\dots,x_{n-1})\in\U{\cD_{(g\m,x_1,\dots,x_{n-1})}\cD_{\eta(x\m_1 g)}\dots\cD_{\eta(x\m_{n-1}\dots x\m_1 g)}}.
$$
 As above, $\A_1\sst\cD_{\eta(x\m_k\dots x\m_1 g)}$, $1\le k\le n-1$, because $\eta(x\m_k\dots x\m_1 g)\in H$. Moreover, by \cref{A_g-sst-D_x-iff-A_1-sst-D_g-inv-x} condition \cref{A_g-in-D_(x_1_dots_x_n)} is equivalent to $\A_1\sst\D_{(g\m,x_1,\dots,x_{n-1})}$. The rest of the proof now follows as for $w'$. If $n=1$, then  
\begin{align}\label{e-for-n=1}
 \e=\prod_{g\in\Lb}\0_g(w(g\m))\in\U\A,
\end{align}
as $\cD_{g\m}\supseteq \A_1$ by \cref{g-in-Lb-iff-A_1-in-D_g-inv} of \cref{Lb-and-Lb'}.
\end{proof}

The following notation will be used in the results below.
\begin{align}
 \Pi(l,m)&=\prod_{k=l,i=m}^{n-1}w\circ\s_{n-1,i}^g(x_1,\dots,x_kx_{k+1},\dots,x_n)^{(-1)^{k+i}}\notag\\
 &\quad\cdot \prod_{i=m}^{n-1}w\circ\s_{n-1,i}^g(x_1,\dots,x_{n-1})^{(-1)^{n+i}},\label{Pi(l_m)-def}
\end{align}
 where $1\le l\le n-1$ and $0\le m\le n-1$. 

\begin{lem}\label{w'-cohom-w-base-0}
 For all $w\in Z^1(G,\A)$ and  $x\in G$ we have:
 \begin{align}\label{delta-e-alpha-inv-w-inv-n=1}
  (\dl^0\e)(x)\alpha_x(1_{x\m}\e)\m w(x)\m= \prod_{g\in\Lb}\0_g(w(g\m x)\m).  
 \end{align}
 Moreover, for $n>1$, $w\in Z^n(G,\A)$ and $x_1,\dots,x_n\in G$:
 \begin{align}
  &(\dl^{n-1}\e)(x_1,\dots,x_n)\alpha_{x_1}(1_{x\m_1}\e(x_2,\dots,x_n))\m w(x_1,\dots,x_n)\m\notag\\
  &= 1_{(x_1,\dots,x_n)}\prod_{g\in\Lb}\0_g(w(g\m x_1,x_2,\dots,x_n)\m\Pi(1,1)).\label{delta-e-alpha-inv-w-inv}
 \end{align}
\end{lem}
\begin{proof}
 Indeed, by \cref{delta^0-def,w=prod_g-0_g-of-w,e-for-n=1} we see that
 \begin{align*}
 (\dl^0\e)(x)\alpha_x(1_{x\m}\e)\m w(x)\m&= \e\m w(x)\m\\ 
 &= \prod_{g\in\Lb} \0_g (w(g\m)\m w(g\m x)\m w(g\m))\\
 &= \prod_{g\in\Lb}\0_g (w(g\m x)\m).
 \end{align*}

For \cref{delta-e-alpha-inv-w-inv} observe from \cref{eps-def,Pi(l_m)-def,delta^n-for-n>0} that
 \begin{align*}
  &(\dl^{n-1}\e)(x_1,\dots,x_n)\alpha_{x_1}(1_{x\m_1}\e(x_2,\dots,x_n))\m\\
  &= \left(\prod_{k=1}^{n-1}\e(x_1,\dots,x_kx_{k+1},\dots,x_n)^{(-1)^k}\right)\e(x_1,\dots,x_{n-1})^{(-1)^n}\\
  &= \prod_{g\in\Lb}\0_g(\Pi(1,0)).
 \end{align*}
Now in \cref{w=prod_g-0_g-of-w} one has
\begin{align*}
 w(g\m,x_1,\dots,x_kx_{k+1},\dots,x_n)^{(-1)^k}&=w\circ\s_{n-1,0}^g(x_1,\dots,x_kx_{k+1},\dots,x_n)^{(-1)^{k+0}},\\
 w(g\m,x_1,\dots,x_{n-1})^{(-1)^n}&=w\circ\s_{n-1,0}^g(x_1,\dots,x_{n-1})^{(-1)^{n+0}},
\end{align*}
 which are the factors of $\Pi(1,0)$ corresponding to $i=0$ and $1\le k\le n-1$. Hence,
$$
 \prod_{g\in\Lb}\0_g(\Pi(1,0))=w(x_1,\dots,x_n)\prod_{g\in\Lb}\0_g(w(g\m x_1,x_2,\dots,x_n)\m\Pi(1,1)).
$$
\end{proof}

\begin{lem}\label{w'-cohom-w-base-1}
 For all $n>1$, $w\in Z^n(G,\A)$, $g\in\Lb$ and $x_1,\dots,x_n\in G$:
 \begin{align}
  w(g\m x_1,x_2,\dots,x_n)\m\Pi(1,1)&=\alpha_{\eta_1^g(x_1)}(1_{\eta_1^g(x_1)\m}w\circ\s_{n-1,0}^{\ol{x\m_1 g}}(x_2,\dots,x_n))\m\notag\\
  &\quad\cdot w(\tau_1^g(x_1),(\ol{x\m_1 g})\m x_2,x_3,\dots,x_n)\m\Pi(2,2)\notag\\
  &\quad\cdot \prod_{i=1}^{n-1}w\circ\s_{n-1,i}^g(x_1x_2,x_3,\dots,x_n)^{(-1)^{i+1}}.\label{w-Pi(1_1)}
 \end{align}
\end{lem}
\begin{proof}
  Since $w$ is a partial $n$-cocycle, one has that (see \cref{eta_n^g-def,tau_n^g-def,sigma_ni-def}) 
\begin{align}
 (\dl^nw)\circ\s_{n,1}^g(x_1,\dots,x_n)&=(\dl^nw)(\tau_1^g(x_1),(\ol{x\m_1 g})\m,x_2,\dots,x_n)\notag\\
 &=(\dl^nw)(\eta_1^g(x_1),(\ol{x\m_1 g})\m,x_2,\dots,x_n)\notag\\
 &=(\dl^nw)(g\m x_1\cdot\ol{x\m_1 g},(\ol{x\m_1 g})\m,x_2,\dots,x_n)\label{delta^nw(g-inv-x_1-ol(x_1-inv-g)_ol(x_1-inv-g)-inv_...)}\\
 &=1_{g\m x_1\cdot\ol{x\m_1 g}}1_{(g\m x_1,x_2,\dots,x_n)}.\notag
\end{align}
  Applying \cref{delta^n-for-n>0}, we expand \cref{delta^nw(g-inv-x_1-ol(x_1-inv-g)_ol(x_1-inv-g)-inv_...)} as follows:
\begin{align*}
 1_{g\m x_1\cdot\ol{x\m_1 g}}1_{(g\m x_1,x_2,\dots,x_n)}&=\alpha_{g\m x_1\cdot\ol{x\m_1 g}}(1_{{(\ol{x\m_1 g})\m x\m_1 g}}w((\ol{x\m_1 g})\m,x_2,\dots,x_n))\\
 &\quad\cdot w(g\m x_1,x_2,\dots,x_n)\m\\
 &\quad\cdot w(g\m x_1\cdot\ol{x\m_1 g},(\ol{x\m_1 g})\m x_2,x_3,\dots,x_n)\\
 &\quad\cdot \prod_{k=2}^{n-1}w(g\m x_1\cdot\ol{x\m_1 g},(\ol{x\m_1 g})\m,x_2,\dots,x_kx_{k+1},\dots x_n)^{(-1)^{k+1}}\\
 &\quad\cdot w(g\m x_1\cdot\ol{x\m_1 g},(\ol{x\m_1 g})\m,x_2,\dots,x_{n-1})^{(-1)^{n+1}}.
\end{align*}
Using our notation \cref{eta_n^g-def,sigma_ni-def,sigma_nn-def,eta_n^g-def,tau_n^g-def}, we conclude that
\begin{align}
 1_{\eta_1^g(x_1)}w(g\m x_1,\dots,x_n)\m&=\alpha_{\eta_1^g(x_1)}(1_{\eta_1^g(x_1)\m}w\circ\s_{n-1,0}^{\ol{x\m_1 g}}(x_2,\dots,x_n))\m\label{af_(eta_1^g(x_1))(1_(eta_1^g(x_1)-inv...)}\\
 &\quad\cdot w(\tau_1^g(x_1),(\ol{x\m_1 g})\m x_2,x_3,\dots,x_n)\m\label{w(tau_1^g(x_1)_ol(x_1-inv-g)-inv-x_2...)}\\
 &\quad\cdot \prod_{k=2}^{n-1}w\circ\s_{n-1,1}^g(x_1,\dots,x_kx_{k+1},\dots x_n)^{(-1)^k}\label{prod-w-circ-s^g_(n-1_1)^(-1)^k}\\
 &\quad\cdot w\circ\s_{n-1,1}^g(x_1,x_2,\dots,x_{n-1})^{(-1)^n},\label{w-circ-s^g_(n-1_1)^(-1)^n}
\end{align}
 the lines \cref{prod-w-circ-s^g_(n-1_1)^(-1)^k,w-circ-s^g_(n-1_1)^(-1)^n} being the inverses of the factors of $\Pi(1,1)$, which correspond to $i=1$ and $2\le k\le n-1$. Thus, after the multiplication of the right-hand side of equality \cref{af_(eta_1^g(x_1))(1_(eta_1^g(x_1)-inv...),w(tau_1^g(x_1)_ol(x_1-inv-g)-inv-x_2...),prod-w-circ-s^g_(n-1_1)^(-1)^k,w-circ-s^g_(n-1_1)^(-1)^n} by $\Pi(1,1)$, they will be reduced, and at their place we shall have the factors of $\Pi(1,1)$ which correspond to $k=1$, and the factors of $\Pi(2,2)$ (i.e. those of $\Pi(1,1)$ with indexes $2\le i,k\le n-1$), giving the right-hand side of \cref{w-Pi(1_1)}. It remains to note that $1_{\eta_1^g(x_1)}\Pi(1,1)=\Pi(1,1)$  and the idempotents which appear in the cancellations are absorbed by the element \cref{w(tau_1^g(x_1)_ol(x_1-inv-g)-inv-x_2...)}, except $1_{x\m_1 g}$ which is absorbed by the element in the right-hand side of \cref{af_(eta_1^g(x_1))(1_(eta_1^g(x_1)-inv...)}.     
\end{proof}

\begin{lem}\label{w'-cohom-w-step}
For all $1<j<n$, $w\in Z^n(G,\A)$, $g\in\Lb$ and $x_1,\dots,x_n\in G$:
\begin{align}
 &w(\tau_{j-1}^g(x_1,\dots,x_{j-1}),(\ol{x\m_{j-1}\dots x\m_1 g})\m x_j,x_{j+1},\dots,x_n)\m\Pi(j,j)\notag\\
 &=\alpha_{\eta_1^g(x_1)}(1_{\eta_1^g(x_1)\m}w\circ\s_{n-1,j-1}^{\ol{x\m_1 g}}(x_2,\dots,x_n))^{(-1)^j}\notag\\
 &\quad\cdot w(\tau_j^g(x_1,\dots,x_j),(\ol{x\m_j\dots x\m_1 g})\m x_{j+1},x_{j+2},\dots,x_n)\m\Pi(j+1,j+1)\notag\\
 &\quad\cdot \prod_{i=j}^{n-1}w\circ\s_{n-1,i}^g(x_1,\dots,x_jx_{j+1},\dots,x_n)^{(-1)^{i+j}}\notag\\
 &\quad\cdot \prod_{s=1}^{j-1}w\circ\s_{n-1,j-1}^g(x_1,\dots,x_sx_{s+1},\dots,x_n)^{(-1)^{s+j}}\label{w-Pi(j_j)}
\end{align}
(here by $\Pi(n,n)$ we mean the identity element $1_\A$).
\end{lem}
\begin{proof}
 We use the same idea as in the proof of \cref{w'-cohom-w-base-1}:
 \begin{align}
  &(\dl^nw)\circ\s_{n,j}^g(x_1,\dots,x_n)\notag\\
  &=(\dl^nw)(\tau_j^g(x_1,\dots,x_j),(\ol{x\m_j\dots x\m_1g})\m,x_{j+1},\dots,x_n)\notag\\
  &=(\dl^nw)(\eta_1^g(x_1),\eta_2^g(x_1,x_2),\dots,\eta_j^g(x_1,\dots,x_j),(\ol{x\m_j\dots x\m_1g})\m,x_{j+1},\dots,x_n)\notag\\
  &=(\dl^nw)(g\m x_1\ol{x\m_1 g},(\ol{x\m_1 g})\m x_2\ol {x\m_2x\m_1 g},\dots,(\ol{x\m_{j-1}\dots x\m_1g})\m x_j \ol{x\m_j\dots x\m_1g},\notag\\
  &\quad\quad(\ol{x\m_j\dots x\m_1g})\m,x_{j+1},\dots,x_n)\label{delta^nw(g-inv-x_1-ol(x_1-inv-g)_ol(x_1-inv-g)-inv-x_2-ol...)}\\
  &=1_{g\m x_1\ol{x\m_1 g}}\, 1_{g\m x_1x_2\ol {x\m_2x\m_1 g}}\dots 1_{g\m x_1\dots x_j\ol{x\m_j\dots x\m_1g}}1_{(g\m x_1\dots x_j,x_{j+1},\dots,x_n)}.\notag
 \end{align}
Expanding \cref{delta^nw(g-inv-x_1-ol(x_1-inv-g)_ol(x_1-inv-g)-inv-x_2-ol...)}, we obtain by \cref{delta^n-for-n>0}
\begin{align*}
 &1_{g\m x_1\ol{x\m_1 g}} \, 1_{g\m x_1x_2\ol {x\m_2x\m_1 g}}\dots 1_{g\m x_1\dots x_j\ol{x\m_j\dots x\m_1g}}1_{(g\m x_1\dots x_j,x_{j+1},\dots,x_n)}\\
 &=\alpha_{g\m x_1\ol{x\m_1 g}}(1_{(\ol{x\m_1 g})\m x\m_1g}w((\ol{x\m_1 g})\m x_2\ol {x\m_2x\m_1 g},\dots,\\
 &\quad\quad(\ol{x\m_{j-1}\dots x\m_1g})\m x_j \ol{x\m_j\dots x\m_1g},(\ol{x\m_j\dots x\m_1g})\m,x_{j+1},\dots,x_n))\\
 &\quad\cdot w(g\m x_1x_2\ol {x\m_2x\m_1 g},\dots,(\ol{x\m_{j-1}\dots x\m_1g})\m x_j \ol{x\m_j\dots x\m_1g},\\
 &\quad\quad(\ol{x\m_j\dots x\m_1g})\m,x_{j+1},\dots,x_n)\m\\
 &\quad\cdot\prod_{s=2}^{j-2}w(g\m x_1\ol{x\m_1 g},\dots,(\ol{x\m_{s-1}\dots x\m_1g})\m x_sx_{s+1} \ol{x\m_{s+1}\dots x\m_1g},\dots ,\\
 &\quad\quad(\ol{x\m_{j-1}\dots x\m_1g})\m x_j \ol{x\m_j\dots x\m_1g},(\ol{x\m_j\dots x\m_1g})\m,x_{j+1},\dots, x_n)^{(-1)^s}\\
 &\quad\cdot w(g\m x_1\ol{x\m_1 g},\dots , (\ol{x\m_{j-1}\dots x\m_1g})\m x_{j-1}x_j \ol{x\m_j\dots x\m_1g},\\
 &\quad\quad(\ol{x\m_j\dots x\m_1g})\m,x_{j+1},\dots, x_n)^{(-1)^{j-1}}\\
 &\quad\cdot w(g\m x_1\ol{x\m_1 g},\dots , (\ol{x\m_{j-2}\dots x\m_1g})\m x_{j-1} \ol{x\m_{j-1}\dots x\m_1g},\\
 &\quad\quad(\ol{x\m_{j-1}\dots x\m_1g})\m x_j,x_{j+1},\dots,x_n)^{(-1)^j}\\
 &\quad\cdot w(g\m x_1\ol{x\m_1 g},\dots , (\ol{x\m_{j-1}\dots x\m_1g})\m x_j \ol{x\m_j\dots x\m_1g},\\
 &\quad\quad(\ol{x\m_j\dots x\m_1g})\m x_{j+1},x_{j+2},\dots,x_n)^{(-1)^{j+1}}\\
 &\quad\cdot\prod_{t=j+1}^{n-1}w(g\m x_1\ol{x\m_1 g},\dots,(\ol{x\m_{j-1}\dots x\m_1g})\m x_j \ol{x\m_j\dots x\m_1g},\\
 &\quad\quad(\ol{x\m_j\dots x\m_1g})\m,x_{j+1},\dots,x_tx_{t+1},\dots,x_n)^{(-1)^{t+1}}\\
 &\quad\cdot w(g\m x_1\ol{x\m_1 g},\dots,(\ol{x\m_{j-1}\dots x\m_1g})\m x_j \ol{x\m_j\dots x\m_1g},\\
 &\quad\quad(\ol{x\m_j\dots x\m_1g})\m,x_{j+1},\dots,x_{n-1})^{(-1)^{n+1}}.
 \end{align*}
 We rewrite this in our shorter notation \cref{eta_n^g-def,sigma_nn-def,sigma_ni-def,eta_n^g-def,tau_n^g-def}:
 \begin{align}
 &1_{\eta_1^g(x_1)}1_{\eta_1^g(x_1x_2)}\dots 1_{\eta_1^g(x_1\dots x_j)}1_{(g\m x_1\dots x_j,x_{j+1},\dots,x_n)}\notag\\
 &=\alpha_{\eta_1^g(x_1)}(1_{\eta_1^g(x_1)\m}w\circ\s_{n-1,j-1}^{\ol{x\m_1 g}}(x_2,\dots,x_n))\notag\\
 &\quad\cdot w\circ\s_{n-1,j-1}^g(x_1x_2,x_3,\dots,x_n)\m\label{w-circ-sigma_n-1_j-1(x_1x_2_...)}\\
 &\quad\cdot\prod_{s=2}^{j-2}w\circ\s_{n-1,j-1}^g(x_1,\dots,x_sx_{s+1},\dots,x_n)^{(-1)^s}\label{prod-w-circ-sigma_n-1_j-1(x_1_..._x_sx_s+1_...)}\\
 &\quad \cdot w\circ\s_{n-1,j-1}^g(x_1,\dots,x_{j-1}x_j,\dots,x_n)^{(-1)^{j-1}}\label{w-circ-sigma_n-1_j-1(x_1_..._x_j-1x_j_...)}\\
 &\quad\cdot w(\tau_{j-1}^g(x_1,\dots,x_{j-1}),(\ol{x\m_{j-1}\dots x\m_1g})\m x_j,x_{j+1},\dots,x_n)^{(-1)^j}\notag\\
 &\quad\cdot w(\tau_j^g(x_1,\dots,x_j),(\ol{x\m_j\dots x\m_1g})\m x_{j+1},x_{j+2},\dots,x_n)^{(-1)^{j+1}}\notag\\
 &\quad\cdot\prod_{t=j+1}^{n-1}w\circ\s_{n-1,j}^g(x_1,\dots,x_tx_{t+1},\dots,x_n)^{(-1)^{t+1}}\notag\\
 &\quad\cdot w\circ\s_{n-1,j}^g(x_1,\dots,x_{n-1})^{(-1)^{n+1}}.\notag
 \end{align}
 Note that the factors \cref{w-circ-sigma_n-1_j-1(x_1x_2_...),w-circ-sigma_n-1_j-1(x_1_..._x_j-1x_j_...)} may be included into the product \cref{prod-w-circ-sigma_n-1_j-1(x_1_..._x_sx_s+1_...)}, permitting thus $s$ to run from $1$ to $j-1$ in \cref{prod-w-circ-sigma_n-1_j-1(x_1_..._x_sx_s+1_...)}. It follows that
 \begin{align}
 &1_{\eta_1^g(x_1\dots x_j)}w(\tau_{j-1}^g(x_1,\dots,x_{j-1}),(\ol{x\m_{j-1}\dots x\m_1g})\m x_j,x_{j+1},\dots,x_n)\m\label{1_eta_1^g(x_1_dots_x_j)-w(...)}\\
 &=\alpha_{\eta_1^g(x_1)}(1_{\eta_1^g(x_1)\m}w\circ\s_{n-1,j-1}^{\ol{x\m_1 g}}(x_2,\dots,x_n))^{(-1)^j}\notag\\
 &\quad\cdot\prod_{s=1}^{j-1}w\circ\s_{n-1,j-1}^g(x_1,\dots,x_sx_{s+1},\dots,x_n)^{(-1)^{s+j}}\notag\\
 &\quad\cdot w(\tau_j^g(x_1,\dots,x_j),(\ol{x\m_j\dots x\m_1g})\m x_{j+1},x_{j+2},\dots,x_n)\m\notag\\
 &\quad\cdot\prod_{t=j+1}^{n-1}w\circ\s_{n-1,j}^g(x_1,\dots,x_tx_{t+1},\dots,x_n)^{(-1)^{t+j+1}}\label{prod-w-circ-sigma^g_n-1_j(x_1_dots_x_tx_t+1_dots_x_n)}\\
 &\quad\cdot w\circ\s_{n-1,j}^g(x_1,\dots,x_{n-1})^{(-1)^{n+j+1}}.\label{1_eta_1^g-w(tau_j-1_...)}
 \end{align}
 The lines  \cref{prod-w-circ-sigma^g_n-1_j(x_1_dots_x_tx_t+1_dots_x_n),1_eta_1^g-w(tau_j-1_...)}  are the inverses of the factors of $\Pi(j,j)$ corresponding to $i=j$ and $j+1\le k\le n-1$. Therefore,  multiplication by $\Pi(j,j)$ replaces  these two lines by the factors of $\Pi(j,j)$ with $j\le i\le n-1$ and $k=j$, and, whenever $j<n-1$, there will also appear all the factors of $\Pi(j+1,j+1)$, giving the right-hand side of equality \cref{w-Pi(j_j)}.  Finally, the left-hand side of \cref{w-Pi(j_j)} coincides with  \cref{1_eta_1^g(x_1_dots_x_j)-w(...)} multiplied by $\Pi(j,j)$, as $1_{\eta_1^g(x_1\dots x_j)}\Pi(j,j)=\Pi(j,j)$.
\end{proof}

\begin{lem}\label{w'-cohom-w-final-step}
 For all $w\in Z^1(G,\A)$, $g\in\Lb$ and $x\in G$:
 \begin{align}\label{w-Pi(n_n)-n=1}
  1_{\eta_1^g(x)}w(g\m x)\m=\alpha_{\eta_1^g(x)}(1_{\eta_1^g(x)\m}w((\ol{x\m g})\m)\m) (w\circ\tau_1^g)(x)\m.
 \end{align}
 Moreover, for all $n>1$, $w\in Z^n(G,\A)$, $g\in\Lb$ and $x_1,\dots,x_n\in G$:
 \begin{align}
  &w(\tau_{n-1}^g(x_1,\dots,x_{n-1}),(\ol{x\m_{n-1}\dots x\m_1 g})\m x_n)\m\notag\\
  &=\alpha_{\eta_1^g(x_1)}(1_{\eta_1^g(x_1)\m}w\circ\s_{n-1,n-1}^{\ol{x\m_1 g}}(x_2,\dots,x_n))^{(-1)^n}\notag\\
  &\quad\cdot\prod_{s=1}^{n-1}w\circ\s_{n-1,n-1}^g(x_1,\dots,x_sx_{s+1},\dots,x_n)^{(-1)^{s+n}}\notag\\
  &\quad\cdot w\circ\tau_n^g(x_1,\dots,x_n)\m.\label{w-Pi(n_n)}
 \end{align}
\end{lem}
\begin{proof}
 For \cref{w-Pi(n_n)-n=1} write
 \begin{align*}
 1_{\eta_1^g(x)}1_{g\m x}&=(\dl^1w) \circ\s_{1,1}^g(x)\\
 &=(\dl^1w)(\tau_1^g(x),(\ol{x\m g})\m)\\
 &=(\dl^1w) (\eta_1^g(x),(\ol{x\m g})\m)\\
 &=\alpha_{\eta_1^g(x)}(1_{\eta_1^g(x)\m}w((\ol{x\m g})\m))w(g\m x)\m (w\circ\tau_1^g)(x).
 \end{align*}

 To get \cref{w-Pi(n_n)}, analyze the proof of \cref{w'-cohom-w-step} (we skip the details):
\begin{align*}
 &1_{\eta_1^g(x_1)}1_{\eta_1^g(x_1x_2)}\dots 1_{\eta_1^g(x_1\dots x_n)}1_{g\m x_1 \cdot \ldots \cdot x_n}\\
 &=(\dl^nw)\circ\s_{n,n}^g(x_1,\dots,x_n)\\  
 &=\alpha_{\eta_1^g(x_1)}(1_{\eta_1^g(x_1)\m}w\circ\s_{n-1,n-1}^{\ol{x\m_1 g}}(x_2,\dots,x_n))\\
 &\quad\cdot\prod_{s=1}^{n-1}w\circ\s_{n-1,n-1}^g(x_1,\dots,x_sx_{s+1},\dots,x_n)^{(-1)^s}\\
 &\quad\cdot w(\tau_{n-1}^g(x_1,\dots,x_{n-1}),(\ol{x\m_{n-1}\dots x\m_1 g})\m x_n)^{(-1)^n}\\
 &\quad\cdot w\circ\tau_n^g(x_1,\dots,x_n)^{(-1)^{n+1}}.
 \end{align*}
\end{proof}

\begin{lem}\label{w'-cohom-w-recursion}
 For all $n>0$, $w\in Z^n(G,\A)$ and $x_1,\dots,x_n\in G$:
 \begin{align}
  &(\dl^{n-1}\e)(x_1,\dots,x_n)\alpha_{x_1}(1_{x\m_1}\e(x_2,\dots,x_n))\m w(x_1,\dots,x_n)\m\notag\\
  &=\prod_{g\in\Lb}\0_g\circ\alpha_{\eta_1^g(x_1)}\left(1_{\eta_1^g(x_1)\m}\prod_{j=0}^{n-1}w\circ\s_{n-1,j}^{\ol{x\m_1 g}}(x_2,\dots,x_n)^{(-1)^{j+1}}\right)\notag\\
  &\quad\cdot w'(x_1\dots,x_n)\m.\label{delta-e-alpha-inv-w-inv=prod-w'-inv}
 \end{align}
\end{lem}
\begin{proof}
 If $n=1$, then the result follows from \cref{delta-e-alpha-inv-w-inv-n=1,w-Pi(n_n)-n=1,w'-def,0_g(a)=0_g(1_xa)} and the fact that $\eta_1^g(x)\in H$.
 
 Let $n>1$. Using the recursion whose base is \cref{w-Pi(1_1)}, an intermediate step is \cref{w-Pi(j_j)} and the final step is \cref{w-Pi(n_n)}, we have
 \begin{align}
  &w(g\m x_1,\dots,x_n)\m\Pi(1,1)\notag\\
  &=\alpha_{\eta_1^g(x_1)}\left(1_{\eta_1^g(x_1)\m}\prod_{j=0}^{n-1}w\circ\s_{n-1,j}^{\ol{x\m_1 g}}(x_2,\dots,x_n)^{(-1)^{j+1}}\right)\notag\\
  &\quad\cdot w\circ\tau_n^g(x_1,\dots,x_n)\m\notag\\
  &\quad\cdot\prod_{j=1}^{n-1}\prod_{i=j}^{n-1}w\circ\s_{n-1,i}^g(x_1,\dots,x_jx_{j+1},\dots,x_n)^{(-1)^{i+j}}\label{prod_j=1^n-1-i=j^n-1}\\
  &\quad\cdot\prod_{j=2}^n\prod_{s=1}^{j-1}w\circ\s_{n-1,j-1}^g(x_1,\dots,x_sx_{s+1},\dots,x_n)^{(-1)^{s+j}}.\label{prod_j=1^n-1-s=1^j-1}
 \end{align}
 After the change of indexes $j'=j-1$ the product \cref{prod_j=1^n-1-s=1^j-1} becomes
 $$
 \prod_{j'=1}^{n-1}\prod_{s=1}^{j'}w\circ\s_{n-1,j'}^g(x_1,\dots,x_sx_{s+1},\dots,x_n)^{(-1)^{s+j'+1}}.
 $$
 Now switching the order in this double product, we come to
 $$
 \prod_{s=1}^{n-1}\prod_{j'=s}^{n-1}w\circ\s_{n-1,j'}^g(x_1,\dots,x_sx_{s+1},\dots,x_n)^{(-1)^{s+j'+1}}.
 $$
 The latter is exactly the inverse of \cref{prod_j=1^n-1-i=j^n-1}. Hence,
 \begin{align*}
  &w(g\m x_1,x_2,\dots,x_n)\m\Pi(1,1)\\
  &=\alpha_{\eta_1^g(x_1)}\left(1_{\eta_1^g(x_1)\m}\prod_{j=0}^{n-1}w\circ\s_{n-1,j}^{\ol{x\m_1 g}}(x_2,\dots,x_n)^{(-1)^{j+1}}\right)\\
  &\quad\cdot w\circ\tau_n^g(x_1,\dots,x_n)\m.
  \end{align*}
 It remains to substitute this into \cref{delta-e-alpha-inv-w-inv} and to apply \cref{w'-def}.
\end{proof}

\begin{lem}\label{apply-af-to-prod-and-switch-with-0}
For all $x\in G$ and $a:\Lb'\to\A$ one has
\begin{align}\label{af_x(1_x-inv-prod-0_g)=1_x-prod-0_g-circ-af}
 \af_x\left(1_{x\m}\prod_{g\in\Lb}\0_g(a(g))\right)=1_x\prod_{g\in\Lb}\0_g\circ\af_{\eta_1^g(x)}\left(1_{\eta_1^g(x)\m}a\left(\ol{x\m g}\right)\right).
\end{align}
\end{lem}
\begin{proof}
 First of all observe using \cref{ol-xg-in-Lb-iff-A_g-in-D_x-inv} of \cref{Lb-and-Lb'} that
\begin{align}\label{1_x-prod_g-in-Lambda}
 1_x\prod_{g\in\Lb}c_g=\prod_{g\in\Lb,\A_g\sst\cD_x}c_g=\prod_{g, \; \ol{x\m g}\in\Lb}c_g,
\end{align}
where $c_g$ is an arbitrary element of $\A_g$. Thus, in the right-hand side of \cref{af_x(1_x-inv-prod-0_g)=1_x-prod-0_g-circ-af} we may replace the condition $g\in\Lb$ by a stronger one $g,\ol{x\m g}\in\Lb$. Notice also from \cref{A_g-sst-D_x-iff-A_1-sst-D_g-inv-x,0_g(a)=0_g(1_xa)} that we may put $1_{g\m x}$ inside of $\0_g$ in the right-hand side of \cref{af_x(1_x-inv-prod-0_g)=1_x-prod-0_g-circ-af}. 

Now
\begin{align}\label{1_(g-inv-x)af_(eta_1^g)=af_(g-inv-x)-circ-af_(ol(g-inv-x))}
 1_{g\m x}\af_{\eta_1^g(x)}\left(1_{\eta_1^g(x)\m}a\left(\ol{x\m g}\right)\right)=\alpha_{g\m x_1}\circ\alpha_{\ol{x\m_1 g}}\left(1_{(\ol{x\m g})\m}1_{\eta_1^g(x)\m}a\left(\ol{x\m g}\right)\right),
\end{align}
and denoting the argument of $\af_{\ol{x\m_1 g}}$ in \cref{1_(g-inv-x)af_(eta_1^g)=af_(g-inv-x)-circ-af_(ol(g-inv-x))} by $b=b(g,x)$, we deduce from \cref{0_g=pr_g-circ-alpha_g} that
\begin{align*}
 \0_g\circ\af_{g\m x}\circ\af_{\ol{x\m g}}(b)&=\pr_g\circ\af_g\left(1_{g\m}\af_{g\m x}\circ\af_{\ol{x\m g}}(b)\right)\\
  &=\pr_g\circ\af_g\circ\af_{g\m}\circ\af_x\left(1_{x\m}1_{x\m g}\af_{\ol{x\m g}}(b)\right)\\
  &=\pr_g\circ\af_x\left(1_{x\m}1_{x\m g}\af_{\ol{x\m g}}(b)\right).
\end{align*}
As $\ol{x\ol{x\m g}}=g\in\Lb$, by \cref{ol-xg-in-Lb-iff-A_g-in-D_x-inv} of \cref{Lb-and-Lb'} we have $\A_{\ol{x\m g}}\sst\cD_{x\m}$ and $\af_x\left(\A_{\ol{x\m g}}\right)=\A_g$. Moreover, $\A_{\ol{x\m g}}\sst\D_{x\m g}$ by \cref{A_g-sst-D_x-iff-A_1-sst-D_g-inv-x}. Hence, in view of \cref{isom-principal-ideals,0_g=pr_g-circ-alpha_g}
\begin{align*}
 \pr_g\circ\af_x\left(1_{x\m}1_{x\m g}\af_{\ol{x\m g}}(b)\right)&=\af_x\circ\pr_{\ol{x\m g}}\left(1_{x\m}1_{x\m g}\af_{\ol{x\m g}}(b)\right)\\
 &=\af_x\circ\pr_{\ol{x\m g}}\circ\af_{\ol{x\m g}}(b)\\
 &=\af_x\circ\0_{\ol{x\m g}}(b),
\end{align*} 
and consequently
$$ 
\0_g\circ\af_{\eta_1^g(x)}\left(1_{\eta_1^g(x)\m}a\left(\ol{x\m g}\right)\right)=\alpha_x\circ\0_{\ol{x\m g}}(b)=\alpha_x\circ\0_{\ol{x\m g}}\left(a\left(\ol{x\m g}\right)\right).
$$
Here we used \cref{0_g(a)=0_g(1_xa)} to remove $1_{\eta_1^g(x)\m}$ and $1_{(\ol{x\m g})\m}$ from $b$. It follows that the right-hand side of \cref{af_x(1_x-inv-prod-0_g)=1_x-prod-0_g-circ-af} is 
\begin{align}
\prod_{g,\ol{x\m g}\in\Lb}\alpha_x\circ\0_{\ol{x\m g}}\left(a\left(\ol{x\m g}\right)\right)=\alpha_x\left(\prod_{g,\ol{x\m g}\in\Lb}\0_{\ol{x\m g}}\left(a\left(\ol{x\m g}\right)\right)\right),\label{af_x(prod-0_(x-inv-g)(1_(ol-x-inv-g)a(ol-x-inv-g)))}
\end{align}
which is verified by checking the projection of each side of the latter equality onto an arbitrary block $\A_t$, $t\in\Lb$. Let $g'=\ol{x\m g}\in\Lb$. Then $g=\ol{x\ol{x\m g}}=\ol{xg'}\in\Lb$, so \cref{af_x(prod-0_(x-inv-g)(1_(ol-x-inv-g)a(ol-x-inv-g)))} becomes, in view of \cref{1_x-prod_g-in-Lambda},
$$
\alpha_x\left(\prod_{g',\ol{xg'}\in\Lb}\0_{g'}\left(a\left(g'\right)\right)\right)=\alpha_x\left(1_{x\m}\prod_{g'\in\Lb}\0_{g'}\left(a\left(g'\right)\right)\right),
$$
proving \cref{af_x(1_x-inv-prod-0_g)=1_x-prod-0_g-circ-af}.
\end{proof}

\begin{lem}\label{alpha_eta_1^g(x_1)}
For all $n>0$, $w\in Z^n(G,\A)$ and $x_1,\dots,x_n\in G$:
\begin{align}
 &1_{(x_1\dots,x_n)}\prod_{g\in\Lb}\0_g\circ\alpha_{\eta_1^g(x_1)}\left(1_{\eta_1^g(x_1)\m}\prod_{j=0}^{n-1}w\circ\s_{n-1,j}^{\ol{x\m_1 g}}(x_2,\dots,x_n)^{(-1)^j}\right)\notag\\
 &=\alpha_{x_1}\left(1_{x\m_1}\e(x_2,\dots,x_n)\right).\label{prod_g-alpha_eta_1^g=alpha_x_1}
\end{align}
\end{lem}
\begin{proof}
Let us fix $n$, $w$ and $x_2,\dots,x_n$. For arbitrary $g\in\Lb'$ define
$$
 a(g)=\prod_{j=0}^{n-1}w\circ\s_{n-1,j}^g(x_2,\dots,x_n)^{(-1)^j}.
$$
Then the left-hand side of \cref{prod_g-alpha_eta_1^g=alpha_x_1} equals
$$
1_{(x_1\dots,x_n)}\prod_{g\in\Lb}\0_g\circ\alpha_{\eta_1^g(x_1)}\left(1_{\eta_1^g(x_1)\m}a(\ol{x\m_1 g})\right).
$$
Since $1_{(x_1\dots,x_n)}=1_{(x_1\dots,x_n)}1_{x_1}$, then applying \cref{apply-af-to-prod-and-switch-with-0}, we transform this into
$$
1_{(x_1\dots,x_n)}\af_{x_1}\left(1_{x\m_1}\prod_{g\in\Lb}\0_g\left(\prod_{j=0}^{n-1}w\circ\s_{n-1,j}^g(x_2,\dots,x_n)^{(-1)^j}\right)\right).
$$
Rewriting $1_{(x_1,\dots,x_n)}$ as $\alpha_{x_1}\left(1_{x\m_1}1_{(x_2,\dots,x_n)}\right)$ and using \cref{eps-def}, we come to the right-hand side of \cref{prod_g-alpha_eta_1^g=alpha_x_1}.
\end{proof}

\begin{thrm}\label{w'-cohom-w}
 Let $\A$ be a direct product of commutative unital indecomposable rings with a structure of a (unital) partial $G$-module, $n>0$, $w\in Z^n(G,\A)$ and $w',\e$ be as in \cref{defn-of-w'-and-eps}. Then $w=\dl^{n-1}\e\cdot w'$. In particular, $w'\in Z^n(G,\A)$.
\end{thrm}	
\begin{proof}
 This is an immediate consequence of \cref{w'-cohom-w-recursion,alpha_eta_1^g(x_1)}.
\end{proof}

\section{Existence of a globalization}\label{existence-glob}

In this section we construct the cocycle $\wtl w$ whose existence was announced above. Keeping the notation of \cref{notion-glob}, we begin with some auxiliary formulas whose proof will be left to the reader.

\begin{lem}\label{formulas-for-eta_n^g}
 Let $g\in\Lb'$. Then
 \begin{align}
  \eta_n^g(x_1,\dots,x_n)&=\eta_{n-1}^{\ol{x\m_1g}}(x_2,\dots,x_n),\ n\ge 2,\label{eta_n^g=eta_n-1^ol(x-inv-g)}\\
  \eta_n^g(x_1,\dots,x_i,x_{i+1},\dots,x_n)&=\eta_{n-1}^g(x_1,\dots,x_ix_{i+1},\dots,x_n),\ 1\le i\le n-2,\label{eta_n^g-with-x_ix_i+1-glued}\\
  \eta_n^g(x_1,\dots,x_{n-1},x_nx_{n+1})&=\eta_n^g(x_1,\dots,x_n)\eta_{n+1}^g(x_1,\dots,x_{n+1}),\ n\ge 1.\label{eta_n^g-eta_n+1^g}
 \end{align}
\end{lem}

We now define a function $\wtl{w'}:G^n\to\A$ by removing the idempotent $1_{(x_1,\dots,x_n)}$ from the right-hand side of \cref{w'-def}, that is
\begin{align}\label{w'-tilde-def}
\wtl{w'}(x_1,\dots,x_n)=\prod_{g\in\Lb}\0_g\circ w\circ\tau_n^g(x_1,\dots,x_n).
\end{align}
As it was observed in the proof of \cref{w'-and-e-are-cochains}, $\wtl{w'}(x_1,\dots,x_n)\in\U\A$, so $\wtl{w'}$ is a classical $n$-cochain from $C^n(G,\U\A)$. It turns out that $\wtl{w'}$ satisfies the ``quasi'' $n$-cocycle identity \cref{cocycle-ident-for-tilde-w}.

\begin{lem}\label{w'-tilde-is-quasi-cocycle}
 Let $n>0$, $w\in Z^n(G,A)$ and $x_1,\dots,x_n\in G$. Then
 \begin{align}
 \af_{x_1}\left(1_{x\m_1}\wtl{w'}(x_2,\dots,x_{n+1})\right)&\prod_{i=1}^n \wtl{w'}(x_1,\dots,x_ix_{i+1},\dots,x_{n+1})^{(-1)^i}\notag\\
 & \cdot\wtl{w'}(x_1,\dots,x_n)^{(-1)^{n+1}} = 1_{x_1}.\label{cocycle-ident-for-tilde-w'}
 \end{align}
\end{lem}
\begin{proof}
 According to \cref{w'-tilde-def}, the left-hand side of \cref{cocycle-ident-for-tilde-w'} is
\begin{align}
&\af_{x_1}\left(1_{x\m_1}\prod_{g\in\Lb}\0_g\circ w\circ\tau_n^g(x_2,\dots,x_{n+1})\right)\label{af_x_1(1_x_1-inv-prod_g-in-Lb)}\\
&\cdot\prod_{i=1}^n\prod_{g\in\Lb}\0_g\circ w\circ\tau_n^g(x_1,\dots,x_ix_{i+1},\dots,x_{n+1})^{(-1)^i}\label{prod_i=1^n-prod_g-in-Lb}\\
&\cdot \prod_{g\in\Lb}\0_g\circ w\circ\tau_n^g(x_1,\dots,x_n)^{(-1)^{n+1}}.\label{prod_g-in-Lb-0_g-circ-w-circ-tau_n^g}
\end{align}
Using \cref{apply-af-to-prod-and-switch-with-0}, we rewrite \cref{af_x_1(1_x_1-inv-prod_g-in-Lb)} as
$$
 1_{x_1}\prod_{g\in\Lb}\0_g\circ\af_{\eta_1^g(x_1)}\left(1_{\eta_1^g(x_1)\m}w\circ\tau_n^{\ol{x\m_1 g}}(x_2,\dots,x_{n+1})\right).
$$
Moreover, since $\0_g$ is a homomorphism, \cref{prod_i=1^n-prod_g-in-Lb} coincides with
$$
\prod_{g\in\Lb}\0_g\left(\prod_{i=1}^n w\circ\tau_n^g(x_1,\dots,x_ix_{i+1},\dots,x_{n+1})^{(-1)^i}\right).
$$
Therefore, in order to prove \cref{cocycle-ident-for-tilde-w'}, it suffices to check the equality
\begin{align}
1_{\eta(x\m_1 g)}\dots 1_{\eta(x\m_{n+1}\dots x\m_1 g)}&=\af_{\eta_1^g(x_1)}\left(1_{\eta_1^g(x_1)\m}w\circ\tau_n^{\ol{x\m_1 g}}(x_2,\dots,x_{n+1})\right)\label{af_eta_1^g(x_1)(1_eta_1^g(x_1)-inv-w-circ-tau_n)}\\
&\quad\cdot\prod_{i=1}^n w\circ\tau_n^g(x_1,\dots,x_ix_{i+1},\dots,x_{n+1})^{(-1)^i}\label{prod_i=1^n-w-circ-tau_n^g(x_1_..._x_ix_i+1_..._x_n+1)}\\
&\quad\cdot w\circ\tau_n^g(x_1,\dots,x_n)^{(-1)^{n+1}}.\label{w-circ-tau_n^g(x_1_..._x_n)}
\end{align}
Indeed, each $\eta(x\m_i\dots x\m_1 g)$ belongs to $H$, so by \cref{0_g-def}
$$
\0_g\left(1_{\eta(x\m_i\dots x\m_1 g)}\right)=\af_g\circ\pr_1\left(1_{\eta(x\m_i\dots x\m_1 g)}\right)=\af_g(1_{\A_1})=1_{\A_g},
$$
and consequently,
$$
\prod_{g\in\Lb}\0_g\left(1_{\eta(x\m_1 g)}\dots 1_{\eta(x\m_{n+1}\dots x\m_1 g)}\right)=\prod_{g\in\Lb}1_{\A_g}=1_\A.
$$

We show that \cref{af_eta_1^g(x_1)(1_eta_1^g(x_1)-inv-w-circ-tau_n),prod_i=1^n-w-circ-tau_n^g(x_1_..._x_ix_i+1_..._x_n+1),w-circ-tau_n^g(x_1_..._x_n)} is exactly the partial $n$-cocycle identity 
\begin{align}\label{delta^nw-circ-ta^g_n=1_eta(x_1-inv-g)...}
(\dl^nw)\circ\tau^g_{n+1}(x_1,\dots,x_{n+1})=1_{\eta(x\m_1 g)}\dots 1_{\eta(x\m_{n+1}\dots x\m_1 g)}.
\end{align}
By \cref{tau_n^g-def,eta_n^g=eta_n-1^ol(x-inv-g)} one has
$$
\tau_n^{\ol{x\m_1 g}}(x_2,\dots,x_{n+1})=(\eta^g_2(x_1,x_2),\dots,\eta^g_{n+1}(x_1,\dots,x_{n+1})),
$$
so the right-hand side of \cref{af_eta_1^g(x_1)(1_eta_1^g(x_1)-inv-w-circ-tau_n)} is the first factor of the left-hand side of \cref{delta^nw-circ-ta^g_n=1_eta(x_1-inv-g)...} expanded in accordance with \cref{delta^n-for-n>0}. Now, the $i$th factor of the product \cref{prod_i=1^n-w-circ-tau_n^g(x_1_..._x_ix_i+1_..._x_n+1)} is of the form
\begin{align*}
w(\tau^g_{i-1}(x_1,\dots,x_{i-1}),\eta^g_i(x_1,\dots,x_{i-1},x_ix_{i+1}),\dots,\eta^g_n(x_1,\dots,x_ix_{i+1},\dots,x_{n+1}))^{(-1)^i},
\end{align*}
which coincides with the $i$th factor of the analogous product of the expansion of the left-hand side of \cref{delta^nw-circ-ta^g_n=1_eta(x_1-inv-g)...} thanks to \cref{eta_n^g-eta_n+1^g,eta_n^g-with-x_ix_i+1-glued}. Finally, \cref{w-circ-tau_n^g(x_1_..._x_n)} is literally the last factor of the above mentioned expansion.
\end{proof}

We proceed now with the construction of $\wtl w$ needed in \cref{w-glob-iff-exists-tilde-w}. Given $n>0$ and $x_1,\dots,x_n\in G$, we define 
\begin{align}\label{tilde-epsilon-def}
\tl\e(x_1,\dots,x_{n-1})=\e(x_1,\dots,x_{n-1})+1_\A-1_{(x_1,\dots,x_{n-1})}\in\U\A,
\end{align} 
understanding that $\tl\e=\e\in\U\A$ if $n=1$. Define also 
\begin{align}\label{tilde-w-in-terms-of-tl-w'} 
\wtl w(x_1,\dots,x_n)=(\tl\dl^{n-1}\tl\e)(x_1,\dots,x_n)\wtl{w'}(x_1,\dots,x_n)\in\U\A,
\end{align} 
where 
\begin{align}\label{tilde-delta-w-resp-to-tl-af} 
(\tl\dl^{n-1}\tl\e)(x_1,\dots,x_n)&=\tl\af_{x_1}(\tl\e(x_2,\dots,x_n))\notag\\  
&\quad\cdot\prod_{i=1}^{n-1}\tl\e(x_1,\dots,x_i x_{i+1},\dots,x_n)^{(-1)^i}\notag\\ 
&\quad\cdot\tl\e(x_1,\dots,x_{n-1})^{(-1)^n},
\end{align}  
and 
\begin{align}\label{tilde-alpha-def} 
\tl\af_x(a)=\af_x(1_{x\m}a)+1_\A-1_x,
\end{align}  
with $x\in G$ and $a\in \A$.

Our main result is as follows.
\begin{thrm}\label{glob-exists} 
Let $\A$ be a commutative unital ring which is a (possibly infinite) direct product of indecomposable rings, and let $\af=\{\af_g:\D_{g\m}\to\D_g\mid g\in G\}$ be a (not necessarily transitive) unital partial action of $G$ on $\A$. Then for any $n\ge 0$ each partial cocycle $w\in Z^n(G,\A)$ is globalizable. 
\end{thrm}
\begin{proof} 
Since the case $n=0$ has been explained in \cref{0-cocycle-is-globalizable}, we assume $n>0$. Consider first the transitive case. We will show that our $\wtl w$ defined in \cref{tilde-w-in-terms-of-tl-w'} satisfies \cref{cocycle-ident-for-tilde-w,w-is-restr-of-tilde-w}. It directly follows from \cref{w'-def,w'-tilde-def,tilde-epsilon-def,tilde-delta-w-resp-to-tl-af,tilde-alpha-def} that 
$$ 
 1_{(x_1,\dots,x_n)}\wtl w(x_1,\dots,x_n)=(\dl^{n-1}\e)(x_1,\dots,x_n)\cdot w'(x_1,\dots,x_n) 
$$
for all $x_1,\dots,x_n\in G$.  By \cref{w'-cohom-w} this yields that $\wtl w$  satisfies  \cref{w-is-restr-of-tilde-w}.  As for \cref{cocycle-ident-for-tilde-w}, we see that 
\begin{align}\label{cboper-tilde-w}
\af_{x_1}\left(1_{x\m_1}\wtl w(x_2,\dots,x_{n+1})\right)&\prod_{i=1}^n\wtl w(x_1,\dots,x_ix_{i+1},\dots,x_{n+1})^{(-1)^i}\notag \\
&\cdot\wtl w(x_1,\dots,x_n)^{(-1)^{n+1}} 
\end{align} 
can be written as product of the following two factors
\begin{align}
\af_{x_1}\left(1_{x\m_1}\wtl{w'}(x_2,\dots,x_{n+1})\right)&\prod_{i=1}^n\wtl{w'} (x_1,\dots,x_ix_{i+1},\dots,x_{n+1})^{(-1)^i}\notag\\
&\cdot\wtl{w'}(x_1,\dots,x_n)^{(-1)^{n+1}}\label{delta-of-tilde-w'}
\end{align}  
and 
\begin{align}\label{cboper-tilde-delta}
\af_{x_1}\left(1_{x\m_1}(\tl\dl^{n-1}\e)(x_2,\dots,x_{n+1})\right)&\prod_{i=1}^n(\tl\dl^{n-1}\e)(x_1,\dots,x_ix_{i+1},\dots,x_{n+1})^{(-1)^i}\notag\\
&\cdot(\tl\dl^{n-1}\e)(x_1,\dots,x_n)^{(-1)^{n+1}}. 
\end{align}
Thanks to \cref{w'-tilde-is-quasi-cocycle} the factor \cref{delta-of-tilde-w'} is $1_{x_1}$, whereas the expansion of \cref{cboper-tilde-delta} has the same form as the usual $\dl^n\circ\dl^{n-1}$ in homological algebra, with the difference that instead of a global action we have a mixture of $\af$ with $\tl\af$. Consequently, all factors in \cref{cboper-tilde-delta} to which neither $\af$, nor $\tl\af$ is applied, cancel amongst themselves resulting in $1_\A$. The remaining factors of the expansion are those of 
\begin{equation}\label{af_x_1(1_x_1-inv(tl-delta^n-1e))}
 \af_{x_1}\left(1_{x\m_1}(\tl\dl^{n-1}\e)(x_2,\dots,x_{n+1})\right)
\end{equation} 
and the first factors in each 
\begin{equation}\label{(tl-delta^n-1-e)(x_1_..._x_ix_i+1_..._x_n+1)}
(\tl\dl^{n-1}\e)(x_1,\dots,x_ix_{i+1},\dots,x_{n+1})^{(-1)^i},\ 1\le i\le n,
\end{equation} 
and in 
\begin{equation}\label{(tl-delta^n-1-e)(x_1_..._x_n)}
(\tl\dl^{n-1}\e)(x_1,\dots,x_n)^{(-1)^{n+1}}.
\end{equation} 
The factors in the expansion of \cref{af_x_1(1_x_1-inv(tl-delta^n-1e))}  are exactly
\begin{equation}\label{af_x_1(1_x_1-inv-tl-af_x_2(tl-e(x_3_..._x_n+1)))}
\af_{x_1}\left(1_{x\m_1}\tl\af_{x_2}(\tl\e(x_3,\dots,x_{n+1}))\right),
\end{equation} 
\begin{equation}\label{af_x_1(1_x_1-inv-tl-e(x_2_..._x_i x_i+1_..._x_n+1))}
 \af_{x_1}\left(1_{x\m_1}\tl\e(x_2,\dots,x_i x_{i+1},\dots,x_{n+1})^{(-1)^{i-1}}\right),\ 2\le i\le n,
\end{equation} 
and
\begin{equation}\label{af_x_1(1_x_1-inv-tl-e(x_2_..._x_n))}
\af_{x_1}\left(1_{x\m_1}\tl\e(x_2,\dots,x_n)^{(-1)^n}\right),
\end{equation} 
whereas the first factors in \cref{(tl-delta^n-1-e)(x_1_..._x_ix_i+1_..._x_n+1),(tl-delta^n-1-e)(x_1_..._x_n)} are
\begin{equation}\label{tl-af_x_1x_2(tl-e(x_3_..._x_n+1))}
\tl\af_{x_1x_2}(\tl\e(x_3,\dots,x_{n+1}))\m, 
\end{equation} 
which comes from the case $i=1$ in \cref{(tl-delta^n-1-e)(x_1_..._x_ix_i+1_..._x_n+1)},
\begin{equation}\label{tl-af_x_1(tl-e(x_2_..._x_ix_i+1_..._x_n+1))}
\tl\af_{x_1}(\tl\e(x_2,\dots,x_ix_{i+1},\dots,x_{n+1}))^{(-1)^i},\ 2\le i\le n,
\end{equation} 
and
\begin{equation}\label{tl-af_x_1(tl-e(x_2_..._x_n))}
\tl\af_{x_1}(\tl\e(x_2,\dots,x_n)).
\end{equation} 
Multiplying the elements in \cref{tl-af_x_1(tl-e(x_2_..._x_ix_i+1_..._x_n+1)),tl-af_x_1(tl-e(x_2_..._x_n))} by $1_{x_1}$ we see that they are canceled with those in \cref{af_x_1(1_x_1-inv-tl-e(x_2_..._x_i x_i+1_..._x_n+1)),af_x_1(1_x_1-inv-tl-e(x_2_..._x_n))}, respectively. Now \cref{af_x_1(1_x_1-inv-tl-af_x_2(tl-e(x_3_..._x_n+1)))} equals $1_{x_1}\tl\af_{x_1x_2}(\tl\e(x_3,\dots,x_{n+1}))$ due to the commutative version of (19) from~\cite{DES2}, so that it cancels with \cref{tl-af_x_1x_2(tl-e(x_3_..._x_n+1))}. It follows that \cref{cboper-tilde-delta} also equals $1_{x_1}$, and we conclude that $\wtl w$ satisfies \cref{cocycle-ident-for-tilde-w}. It remains to apply \cref{w-glob-iff-exists-tilde-w}.

If $\alpha$ is not transitive, then we represent $\A$ as a product of ideals, on each of which $\alpha$ acts transitively, so that the construction of $\tilde w$ reduces to the transitive case by means of the projection on such an ideal (see \cite[Proposition 8.4]{DES2}).
\end{proof}

\section{Uniqueness of a globalization}\label{sec-uniqueness}

Our aim is to show that the globalization of $w$ constructed in \cref{existence-glob} is unique up to cohomological equivalence.

We would like to use item (iii) of \cite[Lemma 8.3]{DES2}, whose proof was not sufficiently well explained. To clarify it, we need some new terminology. Let $\R$ be a ring and $\R_\mu\sst\R$, $\mu\in M$, a collection of its unital ideals. Observe from the definition of a direct product that there is a unique homomorphism $\phi:\R\to\prod_{\mu\in M}\R_\mu$, such that $\phi$ followed by the natural projection $\prod_{\mu\in M}\R_\mu\to\R_{\mu'}$ coincides with the multiplication by $1_{\R_{\mu'}}$ in $\R$ for any $\mu'\in M$. In this situation we say that the homomorphism $\phi$ {\it respects projections}.
\begin{lem}\label{sumIdeals}
Let $\cC$ be a not necessarily unital ring and $\{\cC_\mu\mid\mu\in M\}$ a family of pairwise distinct unital ideals in $\cC$. Suppose that $I$ and $J$ are unital ideals in $\cC$ such that 
\begin{align}\label{I-cong-prod-and-J-cong-prod}
I\cong\prod_{\mu\in M_1}\cC_\mu\ \text{and}\ J\cong\prod_{\mu\in M_2}\cC_\mu,
\end{align} 
where $M_1,M_2\sst M$, $\cC_\mu\sst I$ for all $\mu\in M_1$ and $\cC_{\mu'}\sst J$ for all $\mu'\in M_2$. If the isomorphisms \cref{I-cong-prod-and-J-cong-prod} respect projections, then there is a (unique) isomorphism
\begin{align}\label{I+J-cong-prod_(Om_1-cup-Om_2)}
I+J\cong\prod_{\mu\in M_1\cup M_2}\cC_\mu,
\end{align}
which also respects projections.
\end{lem}
\begin{proof} 
It is readily seen that $I+J$ is a unital ring with unity element $1_I+1_J-1_I1_J$ and $I+J=I\oplus J'$, where $J'=J(1_J-1_I1_J)$. Therefore, the isomorphism $J\cong\prod_{\mu\in M_2}\cC_\mu$ restricts to
$$
J'\cong\prod_{\mu\in M_2\setminus M_1}\cC_\mu\sst\prod_{\mu\in M_2}\cC_\mu
$$
(see \cref{eA=prod_(lb-in-Lb')A_lb}). Then
\begin{align*}
I+J=I\oplus J'\cong\left(\prod_{\mu\in M_1}\cC_\mu\right)\oplus\left(\prod_{\mu\in M_2\setminus M_1}\cC_\mu\right)&\cong\left(\prod_{\mu\in M_1}\cC_\mu\right)\times\left(\prod_{\mu\in M_2\setminus M_1}\cC_\mu\right), 
\end{align*}
the latter being isomorphic to $\prod_{\mu\in M_1\sqcup(M_2\setminus M_1)}\cC_\mu$, which proves \cref{I+J-cong-prod_(Om_1-cup-Om_2)}. Moreover, the isomorphism can be chosen in such a way that it respects projections, provided that the isomorphisms \cref{I-cong-prod-and-J-cong-prod} have this property.
\end{proof}

\begin{prop}\label{B-embeds-into-prod_(g-in-Lb')A_g}
 Let $\A$ be a product $\prod_{g\in\Lb}\A_g$ of not necessarily commutative indecomposable unital rings, $\af$ a transitive unital partial action of $G$ on $\A$ and $(\bt,\B)$ an enveloping action of $(\af,\A)$ with $\A\sst\B$. Then $\B$ embeds as an ideal into $\prod_{g\in\Lb'}\A_g$, where $\Lb'$ was defined before formula \cref{A_g=af(A_1)} and $\A_g$ denotes the ideal\footnote{This does not conflict with \cref{A_g=af(A_1)}, because $\af_g(\A_1)=\A_g\sst\A$ for $g\in\Lb$, so $\bt_g(A_1)=\af_g(\A_1)$.} $\bt_g(\A_1)$ in $\B$. Moreover, $\M(\B)\cong\prod_{g\in\Lb'}\A_g$, and $\bt^*$ is transitive, when seen as an action of $G$ on $\prod_{g\in\Lb'}\A_g$.
\end{prop}
\begin{proof}
 As it was explained before \cref{sumIdeals}, there is a unique homomorphism $\phi:\B\to \prod_{g\in\Lb'}\A_g$, which respects projections. We shall prove that $\phi$ is injective. Since $\B=\sum_{g\in G}\bt_g(\A)$, each element of $\B$ belongs to an ideal $I$ of $\B$ of the form $\sum_{i=1}^k\bt_{x_i}(\A)$, $x_1,\dots,x_k\in G$. Therefore, it suffices to show that the restriction of $\phi$ to any such $I$ is injective. Using (ii) of \cite[Lemma 8.3]{DES2}, we may construct for any $i=1,\dots,k$ an isomorphism
 $$
 \bt_{x_i}(\A)=\bt_{x_i}\left(\prod_{g\in\Lb}\A_g\right)\cong\prod_{g\in\Lb}\bt_{x_i}(\A_g)=\prod_{g\in\Lb}\A_{\ol{x_ig}},
 $$
 which respects projections. Notice that it follows from the definition of $\Lb'$ that the ideals $\A_g$, $g\in\Lb'$, are pairwise distinct. Hence by \cref{sumIdeals} there is an isomorphism
 \begin{align}\label{B-cong-prod_(g-in_Lb'')A_g}
  \psi:I\to\prod_{g\in\Lb''}\A_g,
 \end{align}
 where $\Lb''=\{\ol{x_ig}\mid g\in\Lb,\ i=1,\dots,k\}\sst\Lb'$, and it also respects projections. We claim that the restriction of $\phi$ to $I$ coincides with $\psi$, if one understands the product in the right-hand side of \cref{B-cong-prod_(g-in_Lb'')A_g} as an ideal in $\prod_{g\in\Lb'}\A_g$ (see \cref{eA=prod_(lb-in-Lb')A_lb}). Indeed, for all $g\in\Lb''$ and $b\in I$ one has
 \begin{align*}
  \pr_g\circ\psi(b)=1_{\A_g}b=\pr_g\circ\phi(b),
 \end{align*}
 because $\phi$ and $\psi$ respect projections. Now if $g\in\Lb'\setminus\Lb''$, then $\ol{x\m_ig}\not\in\Lb$ for all $i=1,\dots,k$, since otherwise $g=\ol{x_i\ol{x\m_ig}}\in\Lb''$. Hence, for all $b=\sum_{i=1}^k\beta_{x_i}(a_i)\in I$ ($a_i\in\A$) in view of (ii) of \cite[Lemma 8.3]{DES2}
 \begin{align*}
  \pr_g\circ\phi(b)=1_{\A_g}b=\sum_{i=1}^k\bt_{x_i}\left(1_{\A_{\ol{x\m_ig}}}a_i\right)=\sum_{i=1}^k\bt_{x_i}(0)=0.
 \end{align*} 
 This proves the claim, and thus injectivity of $\phi$. Moreover, since $\phi(I)=\prod_{g\in\Lb''}\A_g$ is an ideal in $\prod_{g\in\Lb'}\A_g$, it follows that $\phi(\B)$ is also an ideal in $\prod_{g\in\Lb'}\A_g$.
 
 Regarding the second statement of the proposition, notice that each element of $\prod_{g\in\Lb'}\A_g$ acts as a multiplier of $\B$, as $\phi(\B)$ is an ideal in $\prod_{g\in\Lb'}\A_g$. Conversely, let $w\in\M(\B)$. Then $w1_{\A_g}=w1_{\A_g}\cdot 1_{\A_g}\in\A_g$ for all $g\in\Lb'$. Define $a\in\prod_{g\in\Lb'}\A_g$ by $\pr_g(a)=w1_{\A_g}$. We need to show that $\phi(wb)=a\phi(b)$ and $\phi(bw)=\phi(b)a$. Indeed, using the fact that $\phi$ respects projections, we get
 \begin{align*}
  \pr_g(\phi(wb))=1_{\A_g}\cdot wb=w1_{\A_g}\cdot 1_{\A_g}b=w1_{\A_g}\cdot\pr_g(\phi(b))=\pr_g(a\phi(b))
 \end{align*}
for all $g\in\Lb'$. Similarly $\pr_g(\phi(bw))=\pr_g(a\phi(b))$ for arbitrary $g\in\Lb'$. The transitivity of $\bt^*$ easily follows from the definition of $\A_g$ for $g\in\Lb'$.
\end{proof}

\begin{thrm}\label{uniqueness-of-glob}
 Let $\A$ be a product $\prod_{g\in\Lb}\A_g$ of commutative indecomposable unital rings, $\af$ a unital partial action of $G$ on $\A$ and $w_i\in Z^n(G,\A)$, $i=1,2$ ($n>0$). Suppose that $(\bt,\B)$ is an enveloping action of $(\af,\A)$ and $u_i\in Z^n(G,\U{\M(\B)})$ is a globalization of $w_i$, $i=1,2$. If $w_1$ is cohomologous to $w_2$, then $u_1$ is cohomologous to $u_2$. In particular any two globalizations of the same partial $n$-cocycle are cohomologous.
\end{thrm}
\begin{proof}
 Let $\af$ be transitive. Thanks to \cref{B-embeds-into-prod_(g-in-Lb')A_g} we may assume, up to an isomorphism, that $\M(\B)=\prod_{g\in\Lb'}\A_g\supseteq\A$. Define
\begin{align}\label{u'_i-def}
 u'_i(x_1,\dots,x_n)=\prod_{g\in\Lb'}\vt_g\circ u_i\circ\tau_n^g(x_1,\dots,x_n),\ i=1,2,
\end{align}
where $\vt_g$ is a homomorphism $\M(\B)\to\M(\B)$ given by
\begin{align}\label{vt_g-def}
 \vt_g=\bt_g\circ\pr_1.
\end{align}
Since $u'_i$ has the same construction as $w'$ from \cref{sec-lemmas} (see \cref{w'-def}), one has by \cref{w'-cohom-w} that $u'_i\in Z^n(G,\U{\M(\B)})$ and $u_i$ is cohomologous to $u'_i$, $i=1,2$. 

It suffices to prove that $u'_1$ is cohomologous to $u'_2$, provided that $w_1$ is cohomologous to $w_2$. Observe, in view of \cref{w-is-restr-of-u}, that for arbitrary $h_1,\dots,h_n\in H$
$$
 \pr_1\circ u_i(h_1,\dots,h_n)=\pr_1\left(u_i(h_1,\dots,h_n)1_{(h_1,\dots,h_n)}\right)=\pr_1\circ w(h_1,\dots,h_n).
$$
Together with \cref{u'_i-def,vt_g-def} this implies that 
\begin{align}\label{u'_i-in-terms-of-w} 
 u'_i(x_1,\dots,x_n)=\prod_{g\in\Lb'}\vt_g\circ w_i\circ\tau_n^g(x_1,\dots,x_n),\ i=1,2.
\end{align}
Let $w_2=w_1\cdot\dl^{n-1}\xi$ for some $\xi\in C^{n-1}(G,\A)$. Since $\vt_g$ is a homomorphism, one immediately sees from \cref{u'_i-in-terms-of-w} that $u'_2=u'_1(\dl^{n-1}\xi)'$, where
$$
 (\dl^{n-1}\xi)'(x_1,\dots,x_n)=\prod_{g\in\Lb'}\vt_g\circ (\dl^{n-1}\xi)\circ\tau_n^g(x_1,\dots,x_n).
$$
We shall show that 
\begin{align}\label{(dl^(n-1)-xi)'=dl^(n-1)-xi'}
 (\dl^{n-1}\xi)'=\dl^{n-1}\xi'
\end{align}
with
$$
 \xi'(x_1,\dots,x_{n-1})=\prod_{g\in\Lb'}\vt_g\circ\xi\circ\tau_{n-1}^g(x_1,\dots,x_{n-1}).
$$
Taking into account the fact that $\vt_g$ is a homomorphism once again and interchanging the left-hand side and the right-hand side of \cref{(dl^(n-1)-xi)'=dl^(n-1)-xi'}, we may reduce \cref{(dl^(n-1)-xi)'=dl^(n-1)-xi'} to
\begin{align}
 &\beta_{x_1}\left(\prod_{g\in\Lb'}\vt_g\circ\xi\circ\tau^g_{n-1}(x_2,\dots,x_n)\right)\notag\\
 &=\prod_{g\in\Lb'}\vt_g\circ \bt_{\eta_1^g(x_1)}\circ\xi(\eta^g_2(x_1,x_2),\dots,\eta^g_n(x_1,\dots,x_n)),\label{beta_x_1(prod_g-in-Lb')=prod_g-in-Ln'}
\end{align}
whose right-hand side is
\begin{align*}
 \prod_{g\in\Lb'}\vt_g\circ \bt_{\eta_1^g(x_1)}\circ\xi\circ\tau^{\ol{x\m_1g}}_{n-1}(x_2,\dots,x_n)
\end{align*}
by \cref{eta_n^g=eta_n-1^ol(x-inv-g)}. Now it is readily seen that \cref{beta_x_1(prod_g-in-Lb')=prod_g-in-Ln'} follows from the global case of \cref{apply-af-to-prod-and-switch-with-0} (with $\af$ and $\0$ replaced by $\bt^*$ and $\vt$, respectively).

The non-transitive case reduces to the transitive one, using the same argument as in \cref{glob-exists}.
\end{proof}

\begin{cor}\label{H^n(G_A)-cong-H^n(G_B)}
 Let $\A$ be a product $\prod_{g\in\Lb}\A_g$ of commutative indecomposable unital rings, $\af$ a unital partial action of $G$ on $\A$ and $(\bt,\B)$ an enveloping action of $(\af,\A)$. Then the partial cohomology group $H^n(G,\A)$ is isomorphic to the classical (global) cohomology group $H^n(G,\U{\M(\B)})$.
\end{cor}
\begin{proof}
	Indeed, when $n>0$, it follows from \cref{uniqueness-of-glob,glob-exists} that there is a well-defined map $\Phi:H^n(G,\A)\to H^n(G,\U{\M(\B)})$ which sends the class of $w\in Z^n(G,\A)$ to the class of a globalization of $w$. The map $\Phi$ is injective as the ``restriction'' \cref{w-is-restr-of-u} commutes with the coboundary operator, so any two $n$-cocycles $u_1,u_2\in Z^n(G,\U{\M(\B)})$ which differ by an $n$-coboundary $v\in B^n(G,\U{\M(\B)})$ restrict to two partial $n$-cocycles from $Z^n(G,\A)$ which differ by the restriction of $v$, the latter being a partial $n$-coboundary from $B^n(G,\A)$. The constructions of $\wtl{w'}$ and $u$ clearly respect products (see \cref{u-def-n>0,w'-tilde-def}), so $\Phi$ is a monomorphism of groups. It is evidently surjective, as any $u\in Z^n(G,\U{\M(\B)})$ restricts to $w\in Z^n(G,\A)$ by means of \cref{w-is-restr-of-u}, and a globalization of $w$ is cohomologous to $u$ thanks to \cref{uniqueness-of-glob}. For the case $n=0$ (which holds in a more general situation) see \cref{H^0(G_A)-cong-H^0(G_M(B))}.
\end{proof}

\section{Example}\label{sec-exm}

In this section we apply our technique from \cref{notion-glob,sec-lemmas,existence-glob} in a concrete example.

Let $G=\langle g\mid g^3=e\rangle$ and $\B=\prod_{i=1}^3\A_i$, where each $\A_i$ is a copy of some commutative indecomposable unital ring $R$. We write an element of $\B$ as a triple $a=(a_1,a_2,a_3)$, where $a_i\in\A_i$, $i=1,2,3$. Consider the ``right shift'' action $\bt$ of $G$ on $\B$ given by
\begin{align*}
\bt_g(a_1,a_2,a_3)=(a_3,a_1,a_2).
\end{align*}
Denote by $\A$ the ideal $\A_1\times\A_2\times\{0_{\A_3}\}$ in $\B$, which henceforth will be identified with $\A_1 \times \A_2$ for notational purposes, and by $\af$ the (admissible) restriction of $\bt$ to $\A$. Then $\af=\{\af_g:\D_{g\m}\to\D_g\mid g\in G\}$ is a partial action of $G$ on $\A$, where $\D_e=\A$, $\af_e=\id_\A$, $\D_{g\m}=\A_1\times\{0_{\A_2}\}$, $\D_g=\{0_{\A_1}\}\times\A_2$ and $\af_g(a,0_{\A_2})=(0_{\A_1},a)$ for all $(a,0_{\A_2})\in\D_{g\m}$. By construction $(\B,\bt)$ is an enveloping action of $(\A,\af)$, the isomorphism between $\B$ and $\sum_{g\in G}\bt_g(\f(\A))\sst \cF(G,\A)$ is given by
\begin{align*}
\psi(a_1,a_2,a_3)=\f(a_1,0_{\A_2})+\bt_g(\f(a_2,0_{\A_2}))+\bt_{g\m}(\f(a_3,0_{\A_2})),
\end{align*} 
where $\f$ and $\bt$ are as in \cref{beta_x(f)_t,f(a)|_t=af_t-inv(1_ta)}. Observe that $\psi(a_1,a_2,a_3)$ is the function on $G$ with the following values:
\begin{align}
\psi(a_1,a_2,a_3)|_e&=(a_1,a_2),\label{psi(a_1_a_2_a_3)-on-e}\\
\psi(a_1,a_2,a_3)|_g&=(a_2,a_3),\label{psi(a_1_a_2_a_3)-on-g}\\
\psi(a_1,a_2,a_3)|_{g\m}&=(a_3,a_1).\label{psi(a_1_a_2_a_3)-on-g-inv}
\end{align}

\begin{exm}\label{exm-w-of-Z_3}
	Let $w\in Z^2(G,\A)$. Then $u\in Z^2(G,\B)$ given by
	\begin{align*}
	u(e,e)&=(w(e,e)_1,w(e,e)_2,w(e,e)_1),\\
	u(e,g)&=(w(e,e)_1,w(e,e)_2,w(e,e)_1),\\
	u(e,g\m)&=(w(e,e)_1,w(e,e)_2,w(e,e)_1),\\
	u(g,e)&=(w(e,e)_1,w(e,e)_1,w(e,e)_2),\\
	u(g\m,e)&=(w(e,e)_2,w(e,e)_1,w(e,e)_1),\\
	u(g,g\m)&=(w(e,e)_1,w(g,g\m)_2,w(e,e)_1),\\
	u(g\m,g)&=(w(g\m,g)_1,w(e,e)\m_2w(e,e)^2_1,w(e,e)_1),\\
	u(g,g)&=(w(e,e)_1,w(g\m,g)_1,w(g\m,g)_1),\\
	u(g\m,g\m)&=(w(e,e)_1,w(g\m,g)\m_1w(e,e)^2_1,w(e,e)_1),
	\end{align*}
	is a globalization of $w$.
\end{exm}
\begin{proof}
	Clearly, $\af$ is transitive, $H=\{e\}$, $\Lb'=\{e,g,g\m\}$, $\Lb=\{e,g\}\sst\Lb'$, $\bar x=x$ for all $x\in G$ and
	\begin{align*}
	\0_e(a_1,a_2)&=(a_1,0_{\A_2}),\\
	\0_g(a_1,a_2)&=(0_{\A_1},a_1).
	\end{align*}
	Then by \cref{w'-def,eps-def} for all $x,y\in G$ 
	\begin{align*}
	w'(x,y)&=1_x1_{xy}\prod_{g\in\Lb}\0_g\circ w(g\m x\cdot\ol{x\m g},(\ol{x\m g})\m y\cdot\ol{y\m x\m g})\\
	&=1_x1_{xy}\prod_{g\in\Lb}\0_g\circ w(e,e)\\
	&=1_x1_{xy}(w(e,e)_1,w(e,e)_1)
	\end{align*}
	and
	\begin{align*}
	\e(x)&=1_x\prod_{g\in\Lb}\0_g(w(g\m,x)w(g\m x\cdot\ol{x\m g},(\ol{x\m g})\m)\m)\\
	&=1_x\prod_{g\in\Lb}\0_g(w(g\m,x)w(e,g\m x)\m)\\
	&=1_x(\0_e(w(e,x)w(e,x)\m)_1,\0_g(w(g\m,x)w(e,g\m x)\m)_2)\\
	&=1_x(1_{\A_1},w(g\m,x)_1w(e,g\m x)\m_1).
	\end{align*}
	It follows from the partial $2$-cocycle identity for $w$ that
	\begin{align}
	w(e,g)&=1_gw(e,e)=(0_{\A_1},w(e,e)_2),\notag\\
	w(e,g\m)&=1_{g\m}w(e,e)=(w(e,e)_1,0_{\A_2}),\notag\\
	w(g,e)&=\af_g(1_{g\m}w(e,e))=(0_{\A_1},w(e,e)_1),\notag\\
	w(g\m,e)&=\af_{g\m}(1_gw(e,e))=(w(e,e)_2,0_{\A_2}),\notag\\
	w(g,g\m)&=(0_{\A_1},w(g\m,g)_1 w(e,e)\m_2 w(e,e)_1).\label{w(g_g-inv)=(0_A_1...)}
	\end{align}
	So, we have
	\begin{align}
	\e(e)&=(1_{\A_1},w(g\m,e)_1w(e,g\m)\m_1)=(1_{\A_1},w(e,e)_2w(e,e)\m_1),\label{eps-of-e}\\
	\e(g)&=(0_{\A_1},w(g\m,g)_1w(e,e)\m_1),\label{eps-of-g}\\
	\e(g\m)&=1_{g\m},\label{eps-of-g-inv}
	\end{align}
	and thus
	\begin{align*}
	(\dl^1\e)(e,e)&=\e(e)=(1_{\A_1},w(e,e)_2w(e,e)\m_1),\\
	(\dl^1\e)(e,g)&=1_g\e(e)=(0_{\A_1},w(e,e)_2w(e,e)\m_1),\\
	(\dl^1\e)(e,g\m)&=1_{g\m}\e(e)=(1_{\A_1},0_{\A_2})=1_{g\m},\\
	(\dl^1\e)(g,e)&=\af_g(1_{g\m}\e(e))=1_g,\\
	(\dl^1\e)(g\m,e)&=\af_{g\m}(1_g\e(e))=(w(e,e)_2w(e,e)\m_1,0_{\A_2}),\\
	(\dl^1\e)(g,g\m)&=\af_g(1_{g\m}\e(g\m))\e(e)\m\e(g)=(0_{\A_1},w(e,e)\m_2 w(g\m,g)_1),\\
	(\dl^1\e)(g\m,g)&=\af_{g\m}(1_g\e(g))\e(e)\m\e(g\m)=(w(g\m,g)_1w(e,e)\m_1,0_{\A_2}).
	\end{align*}
	Since $\D_g\D_{g\m}=\{0_\A\}$, both $w$ and $\dl^1\e$ are zero at $(g,g)$ and $(g\m,g\m)$. Hence, we explicitly see that $w=\dl^1\e\cdot w'$.
	
	Removing $1_x1_{xy}$ from $w'$ as in \cref{w'-tilde-def}, we have for all $x,y\in G$
	\begin{align*}
		\wtl{w'}(x,y)=(w(e,e)_1,w(e,e)_1).
	\end{align*} 
	Furthermore, by \cref{tilde-epsilon-def,eps-of-e,eps-of-g,eps-of-g-inv}
	\begin{align*}
		\tl\e(e)&=\e(e)=(1_{\A_1},w(e,e)_2w(e,e)\m_1),\\
		\tl\e(g)&=\e(g)+1_\A-1_g=(1_{\A_1},w(g\m,g)_1w(e,e)\m_1),\\
		\tl\e(g\m)&=\e(g\m)+1_\A-1_{g\m}=1_\A.
	\end{align*}
	Therefore, by \cref{tilde-delta-w-resp-to-tl-af,tilde-alpha-def}
	\begin{align*}
		(\tl\dl^1\tl\e)(e,e)&=\tl\e(e)=(1_{\A_1},w(e,e)_2w(e,e)\m_1),\\
		(\tl\dl^1\tl\e)(e,g)&=\tl\e(e)=(1_{\A_1},w(e,e)_2w(e,e)\m_1),\\
		(\tl\dl^1\tl\e)(e,g\m)&=\tl\e(e)=(1_{\A_1},w(e,e)_2w(e,e)\m_1),\\
		(\tl\dl^1\tl\e)(g,e)&=\af_g(1_{g\m}\tl\e(e))+1_{g\m}=1_\A,\\
		(\tl\dl^1\tl\e)(g\m,e)&=\af_{g\m}(1_g\tl\e(e))+1_g=(w(e,e)_2w(e,e)\m_1,1_{\A_2}),\\
		(\tl\dl^1\tl\e)(g,g\m)&=\tl\e(e)\m\tl\e(g)=(1_{\A_1},w(e,e)\m_2 w(g\m,g)_1),\\
		(\tl\dl^1\tl\e)(g\m,g)&=(\af_{g\m}(1_g\tl\e(g))+1_g)\tl\e(e)\m=(w(g\m,g)_1w(e,e)\m_1,w(e,e)\m_2w(e,e)_1),\\
		(\tl\dl^1\tl\e)(g,g)&=\tl\e(g)=(1_{\A_1},w(g\m,g)_1w(e,e)\m_1),\\
		(\tl\dl^1\tl\e)(g\m,g\m)&=\tl\e(g)\m=(1_{\A_1},w(g\m,g)\m_1w(e,e)_1).
	\end{align*}
	Thus, by \cref{tilde-w-in-terms-of-tl-w'}
	\begin{align*}
		\wtl w(e,e)&=(w(e,e)_1,w(e,e)_2)=w(e,e),\\
		\wtl w(e,g)&=(w(e,e)_1,w(e,e)_2)=w(e,e),\\
		\wtl w(e,g\m)&=(w(e,e)_1,w(e,e)_2)=w(e,e),\\
		\wtl w(g,e)&=(w(e,e)_1,w(e,e)_1),\\
		\wtl w(g\m,e)&=(w(e,e)_2,w(e,e)_1),\\
		\wtl w(g,g\m)&=(w(e,e)_1,w(e,e)\m_2 w(g\m,g)_1w(e,e)_1)=(w(e,e)_1,w(g,g\m)_2), & \text{(by \cref{w(g_g-inv)=(0_A_1...)})}\\
		\wtl w(g\m,g)&=(w(g\m,g)_1,w(e,e)\m_2w(e,e)^2_1),\\
		\wtl w(g,g)&=(w(e,e)_1,w(g\m,g)_1),\\
		\wtl w(g\m,g\m)&=(w(e,e)_1,w(g\m,g)\m_1w(e,e)^2_1).
	\end{align*}
	
	Finally, to calculate $u$, we shall use \cref{u-def-n>0,psi(a_1_a_2_a_3)-on-e,psi(a_1_a_2_a_3)-on-g,psi(a_1_a_2_a_3)-on-g-inv}:
	\begin{align*}
		\psi(u(e,e))|_e&=\wtl w(e,e)\wtl w(e,e)\wtl w(e,e)\m=\wtl w(e,e)=w(e,e)=(w(e,e)_1,w(e,e)_2),\\
		\psi(u(e,e))|_g&=\wtl w(g\m,e)\wtl w(g\m,e)\wtl w(g\m,e)\m=\wtl w(g\m,e)=(w(e,e)_2,w(e,e)_1),\\
		\psi(u(e,e))|_{g\m}&=\wtl w(g,e)\wtl w(g,e)\wtl w(g,e)\m=\wtl w(g,e)=(w(e,e)_1,w(e,e)_1),
	\end{align*}
	whence
	\begin{align*}
		u(e,e)=(w(e,e)_1,w(e,e)_2,w(e,e)_1).
	\end{align*}
	\begin{align*}
	\psi(u(e,g))|_e&=\wtl w(e,e)\wtl w(e,g)\wtl w(e,g)\m=\wtl w(e,e)=w(e,e)=(w(e,e)_1,w(e,e)_2),\\
	\psi(u(e,g))|_g&=\wtl w(g\m,e)\wtl w(g\m,g)\wtl w(g\m,g)\m=\wtl w(g\m,e)=(w(e,e)_2,w(e,e)_1),\\
	\psi(u(e,g))|_{g\m}&=\wtl w(g,e)\wtl w(g,g)\wtl w(g,g)\m=\wtl w(g,e)=(w(e,e)_1,w(e,e)_1),
	\end{align*}
	whence
	\begin{align*}
	u(e,g)=(w(e,e)_1,w(e,e)_2,w(e,e)_1).
	\end{align*}
	\begin{align*}
	\psi(u(e,g\m))|_e&=\wtl w(e,e)\wtl w(e,g\m)\wtl w(e,g\m)\m=\wtl w(e,e)=w(e,e)=(w(e,e)_1,w(e,e)_2),\\
	\psi(u(e,g\m))|_g&=\wtl w(g\m,e)\wtl w(g\m,g\m)\wtl w(g\m,g\m)\m=\wtl w(g\m,e)=(w(e,e)_2,w(e,e)_1),\\
	\psi(u(e,g\m))|_{g\m}&=\wtl w(g,e)\wtl w(g,g\m)\wtl w(g,g\m)\m=\wtl w(g,e)=(w(e,e)_1,w(e,e)_1),
	\end{align*}
	whence
	\begin{align*}
	u(e,g\m)=(w(e,e)_1,w(e,e)_2,w(e,e)_1).
	\end{align*}
	\begin{align*}
	\psi(u(g,e))|_e&=\wtl w(e,g)\wtl w(g,e)\wtl w(e,g)\m=\wtl w(g,e)=(w(e,e)_1,w(e,e)_1),\\
	\psi(u(g,e))|_g&=\wtl w(g\m,g)\wtl w(e,e)\wtl w(g\m,g)\m=\wtl w(e,e)=(w(e,e)_1,w(e,e)_2),\\
	\psi(u(g,e))|_{g\m}&=\wtl w(g,g)\wtl w(g\m,e)\wtl w(g,g)\m=\wtl w(g\m,e)=(w(e,e)_2,w(e,e)_1),
	\end{align*}
	whence
	\begin{align*}
	u(g,e)=(w(e,e)_1,w(e,e)_1,w(e,e)_2).
	\end{align*}
	\begin{align*}
	\psi(u(g\m,e))|_e&=\wtl w(e,g\m)\wtl w(g\m,e)\wtl w(e,g\m)\m=\wtl w(g\m,e)=(w(e,e)_2,w(e,e)_1),\\
	\psi(u(g\m,e))|_g&=\wtl w(g\m,g\m)\wtl w(g,e)\wtl w(g\m,g\m)\m=\wtl w(g,e)=(w(e,e)_1,w(e,e)_1),\\
	\psi(u(g\m,e))|_{g\m}&=\wtl w(g,g\m)\wtl w(e,e)\wtl w(g,g\m)\m=\wtl w(e,e)=(w(e,e)_1,w(e,e)_2),
	\end{align*}
	whence
	\begin{align*}
	u(g\m,e)=(w(e,e)_2,w(e,e)_1,w(e,e)_1).
	\end{align*}
	\begin{align*}
	\psi(u(g,g\m))|_e&=\wtl w(e,g)\wtl w(g,g\m)\wtl w(e,e)\m\\
	&=(w(e,e)_1,w(g,g\m)_2)\\
	\psi(u(g,g\m))|_g&=\wtl w(g\m,g)\wtl w(e,g\m)\wtl w(g\m,e)\m\\
	&=(w(g\m,g)_1,w(e,e)\m_2w(e,e)^2_1)\\
	&\quad\cdot(w(e,e)_1,w(e,e)_2)\\
	&\quad\cdot(w(e,e)\m_2,w(e,e)\m_1)\\
	&=(w(g,g\m)_2,w(e,e)_1), & \text{(by \cref{w(g_g-inv)=(0_A_1...)})}\\
	\psi(u(g,g\m))|_{g\m}&=\wtl w(g,g)\wtl w(g\m,g\m)\wtl w(g,e)\m\\
	&=(w(e,e)_1,w(g\m,g)_1)\\
	&\quad\cdot(w(e,e)_1,w(g\m,g)\m_1w(e,e)^2_1)\\
	&\quad\cdot(w(e,e)\m_1,w(e,e)\m_1)\\
	&=(w(e,e)_1,w(e,e)_1),
	\end{align*}
	whence
	\begin{align*}
	u(g,g\m)=(w(e,e)_1,w(g,g\m)_2,w(e,e)_1).
	\end{align*}
	\begin{align*}
	\psi(u(g\m,g))|_e&=\wtl w(e,g\m)\wtl w(g\m,g)\wtl w(e,e)\m\\
	&=(w(g\m,g)_1,w(e,e)\m_2w(e,e)^2_1),\\
	\psi(u(g\m,g))|_g&=\wtl w(g\m,g\m)\wtl w(g,g)\wtl w(g\m,e)\m\\
	&=(w(e,e)_1,w(g\m,g)\m_1w(e,e)^2_1)\\
	&\quad\cdot(w(e,e)_1,w(g\m,g)_1)\\
	&\quad\cdot(w(e,e)\m_2,w(e,e)\m_1)\\
	&=(w(e,e)\m_2w(e,e)^2_1,w(e,e)_1),\\
	\psi(u(g\m,g))|_{g\m}&=\wtl w(g,g\m)\wtl w(e,g)\wtl w(g,e)\m\\
	&=(w(e,e)_1,w(g,g\m)_2)\\
	&\quad\cdot(w(e,e)_1,w(e,e)_2)\\
	&\quad\cdot(w(e,e)\m_1,w(e,e)\m_1)\\
	&=(w(e,e)_1,w(g\m,g)_1), & \text{(by \cref{w(g_g-inv)=(0_A_1...)})}
	\end{align*}
	whence
	\begin{align*}
	u(g\m,g)=(w(g\m,g)_1,w(e,e)\m_2w(e,e)^2_1,w(e,e)_1).
	\end{align*}
	\begin{align*}
	\psi(u(g,g))|_e&=\wtl w(e,g)\wtl w(g,g)\wtl w(e,g\m)\m\\
	&=(w(e,e)_1,w(g\m,g)_1),\\
	\psi(u(g,g))|_g&=\wtl w(g\m,g)\wtl w(e,g)\wtl w(g\m,g\m)\m\\
	&=(w(g\m,g)_1,w(e,e)\m_2w(e,e)^2_1)\\
	&\quad\cdot(w(e,e)_1,w(e,e)_2)\\
	&\quad\cdot(w(e,e)\m_1,w(g\m,g)_1w(e,e)^{-2}_1)\\
	&=(w(g\m,g)_1,w(g\m,g)_1),\\
	\psi(u(g,g))|_{g\m}&=\wtl w(g,g)\wtl w(g\m,g)\wtl w(g,g\m)\m\\
	&=(w(e,e)_1,w(g\m,g)_1)\\
	&\quad\cdot(w(g\m,g)_1,w(e,e)\m_2w(e,e)^2_1)\\
	&\quad\cdot(w(e,e)\m_1,w(g,g\m)\m_2)\\
	&=(w(g\m,g)_1,w(e,e)_1), & \text{(by \cref{w(g_g-inv)=(0_A_1...)})}
	\end{align*}
	whence
	\begin{align*}
	u(g,g)=(w(e,e)_1,w(g\m,g)_1,w(g\m,g)_1).
	\end{align*}
	\begin{align*}
	\psi(u(g\m,g\m))|_e&=\wtl w(e,g\m)\wtl w(g\m,g\m)\wtl w(e,g)\m\\
	&=(w(e,e)_1,w(g\m,g)\m_1w(e,e)^2_1),\\
	\psi(u(g\m,g\m))|_g&=\wtl w(g\m,g\m)\wtl w(g,g\m)\wtl w(g\m,g)\m\\
	&=(w(e,e)_1,w(g\m,g)\m_1w(e,e)^2_1)\\
	&\quad\cdot(w(e,e)_1,w(g,g\m)_2)\\
	&\quad\cdot(w(g\m,g)\m_1,w(e,e)_2w(e,e)^{-2}_1)\\
	&=(w(g\m,g)\m_1w(e,e)^2_1,w(e,e)_1), & \text{(by \cref{w(g_g-inv)=(0_A_1...)})}\\
	\psi(u(g\m,g\m))|_{g\m}&=\wtl w(g,g\m)\wtl w(e,g\m)\wtl w(g,g)\m\\
	&=(w(e,e)_1,w(g,g\m)_2)\\
	&\quad\cdot(w(e,e)_1,w(e,e)_2)\\
	&\quad\cdot(w(e,e)\m_1,w(g\m,g)\m_1)\\
	&=(w(e,e)_1,w(e,e)_1), & \text{(by \cref{w(g_g-inv)=(0_A_1...)})}
	\end{align*}
	whence
	\begin{align*}
	u(g\m,g\m)=(w(e,e)_1,w(g\m,g)\m_1w(e,e)^2_1,w(e,e)_1).
	\end{align*}
\end{proof}

	\begin{rem}\label{H^2(Z_3_A)-triv}
		The groups $H^2(G,\A)$ and $H^2(G,\B)$ are trivial.
	\end{rem}
\begin{proof}
	Let $w\in Z^2(G,\A)$. Without loss of generality, we may assume $w$ to be normalized (see~\cite[Remark~2.6]{DK}), i.e. $w(e,e)=1_\A$, $w(e,g)=w(g,e)=1_g$ and $w(e,g\m)=w(g\m,e)=1_{g\m}$. Take $\e(e)=1_\A$, $\e(g)=1_g$ and $\e(g\m)=w(g\m,g)\m$. Then $v:=w\cdot\dl^1\e$ is also normalized and satisfies additionally $v(g\m,g)=1_{g\m}$. Writing the partial $2$-cocycle identity for $v$ with the triple $(g,g\m,g)$, we obtain $v(g,g\m)=\af_g(v(g\m,g))=1_g$. Since also $v(g,g)$ and $v(g\m,g\m)$ belong to $\D_g\D_{g\m}=\{0_\A\}$, we conclude that $v$ is trivial.
	
	Let $u\in Z^2(G,\B)$. As in the partial case, multiplying $u$ by a suitable coboundary, we may make 
	\begin{align}\label{u(e_e)=...=u(g_g-inv)=1_B}
		u(e,e)=u(e,g)=u(g,e)=u(e,g\m)=u(g\m,e)=u(g\m,g)=u(g,g\m)=1_\B.
	\end{align} 
	Now, the $2$-cocycle identity for $u$ written with the triple $(g,g,g)$ gives $\bt_g(u(g,g))=u(g,g)$, so that $u(g,g)_1=u(g,g)_2=u(g,g)_3$. Furthermore, the same identity with $(g\m,g,g)$ implies $u(g\m,g\m)=\bt_{g\m}(u(g,g))\m=u(g,g)\m$. Thus, it suffices to make $u(g,g)=1_\B$ maintaining the conditions \cref{u(e_e)=...=u(g_g-inv)=1_B}. Take $\e(e)=1_\B$, $\e(g)=(1_{\A_1},1_{\A_2},u(g,g)\m_1)$ and $\e(g\m)=(1_{\A_1},u(g,g)_1,1_{\A_3})$. Then $v:=u\cdot\dl^1\e$ is normalized,
	\begin{align*}
		v(g\m,g)=u(g\m,g)\bt_{g\m}(\e(g))\e(e)\m\e(g\m)=\bt_{g\m}(\e(g))\e(g\m)=1_\B,
	\end{align*}
	so that $v(g,g\m)=\bt_g(v(g\m,g))=1_\B$. Finally,
	\begin{align*}
		v(g,g)&=u(g,g)\bt_g(\e(g))\e(g\m)\m\e(g)\\
		&=u(g,g)\bt_g(1_{\A_1},1_{\A_2},u(g,g)\m_1)(1_{\A_1},u(g,g)_1,1_{\A_3})\m(1_{\A_1},1_{\A_2},u(g,g)\m_1)\\
		&=u(g,g)(u(g,g)\m_1,u(g,g)\m_1,u(g,g)\m_1)\\
		&=1_\B,
	\end{align*}
	so that $v(g\m,g\m)=v(g,g)\m=1_\B$ too.
\end{proof}
\section*{Acknowledgments}
The first two authors would like to express their sincere gratitude to the Department of Mathematics of the University of Murcia for its warm hospitality during their visits. We are also grateful to the referee who has pointed out numerous inaccuracies throughout the text, proposed various improvements in the exposition and gave a suggestion to add an example, which resulted in a new section of the paper.

\bibliography{bibl-pact}{}
\bibliographystyle{acm}

\end{document}